\documentclass[11pt, letterpaper]{amsart}

\usepackage{amsmath}
\usepackage{amssymb}
\usepackage{amsthm}
\usepackage[pdftex]{graphicx}
\usepackage{amscd}
\usepackage{comment}
\usepackage{setspace}
\usepackage{url}                 

\begin{document}

\numberwithin{equation}{section}

\theoremstyle{definition}
\newtheorem{theorem}{Theorem}[section]
\newtheorem{proposition}[theorem]{Proposition}
\newtheorem{lemma}[theorem]{Lemma}
\newtheorem{corollary}[theorem]{Corollary}
\newtheorem{remark}[theorem]{Remark}
\newtheorem{definition}[theorem]{Definition}
\newtheorem{example}[theorem]{Example}
\newtheorem{problem}[theorem]{Problem}
\newtheorem{convention}[theorem]{Convention}
\newtheorem{conjecture}[theorem]{Conjecture}

\title[Acylindrical hyperbolicity of Artin groups]
{Acylindrical hyperbolicity and the centers of 
Artin groups that are not free of infinity}

\author[M. Kato]{Motoko Kato}
\thanks{The first author is supported by JSPS KAKENHI Grant Number 20K14311, 
and JST, ACT-X Grant Number JPMJAX200A, Japan.}
\address[M. Kato]{Faculty of Education, University of the Ryukyus, 1 Sembaru, 
Nishihara-Cho, Nakagami-Gun, Okinawa 903-0213, Japan}
\email{katom@cs.u-ryukyu.ac.jp}

\author[S. Oguni]{Shin-ichi Oguni}
\thanks{The second author is supported by JSPS KAKENHI Grant Number 20K03590.}
\address[S. Oguni]{Graduate School of Science and Engineering, Ehime University, 2-5 Bunkyo-cho, Matsuyama, Ehime 790-8577, Japan}
\email{oguni.shinichi.mb@ehime-u.ac.jp}

\keywords{Artin group, Acylindrical hyperbolicity, WPD contracting element, CAT(0)-cube complex, Center} 
\subjclass[2000]{20F36, 20F65, 20F67}

\maketitle

\begin{abstract}    
Charney and Morris-Wright showed acylindrical hyperbolicity of 
Artin groups of infinite type associated with graphs that are not joins, 
by studying clique-cube complexes and the actions on them. 
The authors developed their study and clarified when 
acylindrical hyperbolicity holds for Artin groups of infinite type 
associated with graphs that are not cones. 
In this paper, we introduce reduced clique-cube complexes. 
Using these complexes, we show acylindrical hyperbolicity of irreducible Artin groups 
associated with graphs that are not cliques, i.e., irreducible 
Artin groups that are not free of infinity. 
One application of our main result is that the centers of such Artin groups are finite, 
and that actually they are trivial in many cases.
Another corollary is that irreducible Artin groups of infinite type and of type FC are acylindrically hyperbolic.
\end{abstract}

\section{Introduction}

Artin groups, also called Artin-Tits groups, have been widely studied since
their introduction by Tits [36]. 
Let $\Gamma$ be a finite simple graph 
with the vertex set $V(\Gamma)$ and the edge set $E(\Gamma)$. 
Each edge $e$ has two endpoints, which we denote by $s_e$ and $t_e$. 
We suppose that any edge $e$ is labeled by an integer $\mu(e)\ge 2$. 
The \textit{Artin group} $A_{\Gamma}$ associated with $\Gamma$ is defined by the following presentation:
	\begin{align}\label{pres_standard_eq}
	A_{\Gamma}=\langle V(\Gamma)\mid \underbrace{s_e t_e s_e t_e \cdots}_{\text{length $\mu(e)$}}	
	=\underbrace{t_e s_e t_e s_e\cdots}_{\text{length $\mu(e)$}} \quad \text{$\left(e\in E(\Gamma)\right)$ } 	\rangle.
	\end{align}
Here, $\Gamma$ is called the \textit{defining graph} of $A_{\Gamma}$
and the \textit{rank} of $A_{\Gamma}$ is defined as $\#V(\Gamma)$.
Free abelian groups, free groups, and braid groups are typical examples of Artin groups.
For any $U\subset V(\Gamma)$, the subgroup $A_{U}$ generated by $U$ is called a \textit{standard parabolic subgroup}.
Adding the relation $v^2 = 1$ for each $v\in V(\Gamma)$ to the presentation $(\ref{pres_standard_eq})$
yields the associated \textit{Coxeter group} $W_\Gamma$.
Let $\Gamma^t$ be the \textit{Coxeter graph} defined by the presentation of $W_{\Gamma}$.
That is, $\Gamma^t$ is a finite simple graph with edge labels in $\mathbb{Z}_{\geq 3}\cup \{\infty\}$, 
which is obtained from $\Gamma$ by removing all edges with label $2$ and inserting an edge with label $\infty$
between every pair of vertices that are not adjacent in $\Gamma$.
Although $\Gamma^t$ is used in research related to Coxeter groups and traditional treatments of Artin groups, 
we mainly use $\Gamma$ following many recent studies on Artin groups in geometric group theory. 
We use the associated Coxeter graph $\Gamma^t$ as an aid in this paper. 
See Figure~\ref{example_fig}.
\begin{figure}
\begin{center}
\includegraphics[width=8cm,pagebox=cropbox,clip]{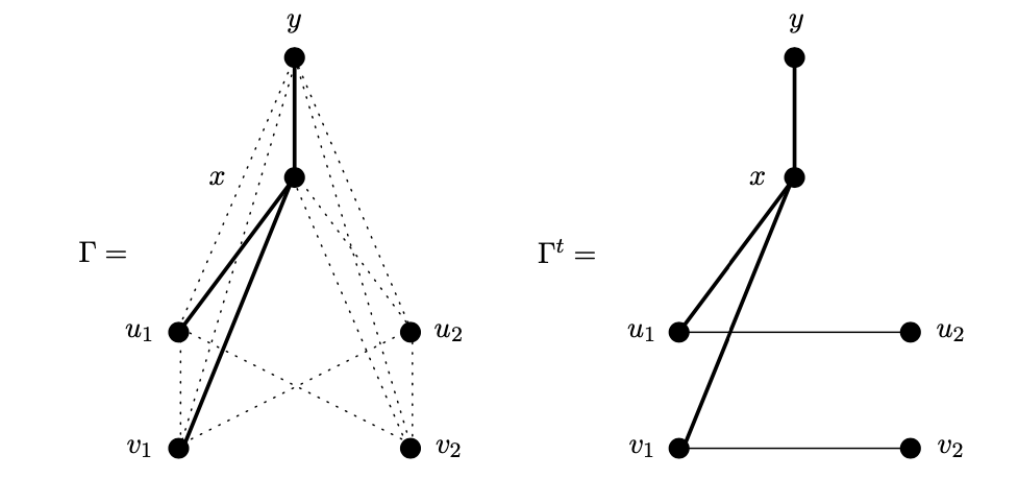} 
\caption{An example of a defining graph $\Gamma$ and the associated Coxeter graph $\Gamma^t$. 
Thick solid lines indicate edges labeled by $3$, and dotted lines indicate edges labeled by $2$. In $\Gamma^t$, thin solid lines denote edges labeled by $\infty$.}\label{example_fig}
\end{center}
\end{figure}

In terms of the properties of $W_{\Gamma}$ or $\Gamma^t$, 
we can define several important classes of Artin groups. 
The Artin group $A_\Gamma$ is said to be \textit{of spherical type} 
if $W_\Gamma$ is finite. Otherwise, it is said to be \textit{of infinite type}. 
The Artin group $A_\Gamma$ is said to be \textit{irreducible} 
if $W_\Gamma$ is \textit{irreducible}; that is, 
when $\Gamma^t$ is connected. 
If $\Gamma^t$ does not contain edges labeled by $\infty$, then we say $A_\Gamma$ is \textit{free of infinity}.
By definition, $A_\Gamma$ is free of infinity if and only if $\Gamma$ is a clique.
In this paper, we focus on Artin groups that are not free of infinity; that is, Artin groups associated with graphs that are not cliques. 

In general, compared with Artin groups, Coxeter groups are relatively well understood. 
For example, 
it is well known that an infinite Coxeter group $W_\Gamma$ is irreducible if and only if 
$W_\Gamma$ cannot be directly decomposed into two nontrivial subgroups (\cite{MR2240393} and \cite{MR2333366}). 
However, it is unclear 
whether $A_\Gamma$ is irreducible if and only if $A_\Gamma$ cannot be directly decomposed 
into two nontrivial subgroups. 
We remark that many basic questions about Artin groups remain open 
(refer to \cite{Charney2008PROBLEMSRT} and \cite{MR3203644}). 

To study Artin groups, 
it is often effective to investigate their hidden
nonpositively curved or negatively curved properties. 
In this paper, we consider 
the following problem \cite[Conjecture B]{haettel2019xxl}. 
\begin{problem}\label{acyl_conj}
Are irreducible Artin groups of infinite type acylindrically hyperbolic?
\end{problem}
\noindent
Acylindrical hyperbolicity is an important property in geometric group theory. 
We can find various applications of acylindrical hyperbolicity in \cite{MR3589159}, \cite{Osin} and \cite{Osin2} for example. 

It is known that reducible Artin groups are not acylindrically hyperbolic, since acylindrically hyperbolic groups cannot be directly decomposed into two infinite subgroups (\cite[Corollary 7.3]{Osin}). 
We also note that Artin groups with infinite centers are not acylindrically hyperbolic, since acylindrically hyperbolic groups do not admit infinite centers (\cite[Corollary 7.3]{Osin}).
Although irreducible Artin groups of spherical type have infinite cyclic center (\cite{BS} and \cite{MR422673})
and thus are not acylindrically hyperbolic,
their central quotients are acylindrically hyperbolic when the rank is at least $2$  (see \cite{MR1914565}, \cite{MR2390326}, \cite{MR2367021} for braid groups
and \cite{MR3719080} for the general case). 

Many affirmative partial answers to Problem~\ref{acyl_conj} are known. 
Indeed, the following irreducible Artin groups of infinite type are known to be 
acylindrically hyperbolic: 
\begin{itemize}
\item Right-angled Artin groups (\cite{MR2827012}, \cite{MR3192368}).
\item Two-dimensional Artin groups such that the associated Coxeter groups are hyperbolic (\cite{alex2019acylindrical}).
\item Artin groups of XXL-type (\cite{haettel2019xxl}).
\item Artin groups of type FC such that the defining graphs have diameter greater than $2$ (\cite{MR3966610}). Here, an Artin group is \textit{of type FC} if every standard parabolic subgroup associated with a clique in the defining graph is of spherical type.
\item Artin groups that are known to be CAT(0) groups according to the result of Brady and McCammond (\cite{MR1770639} and \cite{KO}).
\item Euclidean Artin groups (\cite{Calvez}).
\item \textit{Two-dimensional} Artin groups; that is, 
Artin groups for which every triangle with three edges $e_1,e_2,e_3$ 
in the defining graphs satisfy $\frac{1}{\mu(e_1)} + \frac{1}{\mu(e_2)}  + \frac{1}{\mu(e_3)} \le 1$ 
(\cite{Vaskou}).
\item Artin groups of which defining graphs that are not joins (\cite{Charney}). 
\item Artin groups of which defining graphs that are not cones (\cite{KO2}).
\item Relatively extra-large-type Artin groups under mild conditions (\cite{Goldman}).
\item Even Artin groups of type FC, as well as Artin groups of type FC that visually split over standard parabolic subgroups of spherical type (\cite{CMM}).
\end{itemize}

As a general strategy for Problem~\ref{acyl_conj}, 
we separately consider the case in which the defining graph is a clique 
and the case in which the defining graph is not a clique. 
Note that Artin groups in the latter case are automatically of infinite type. 
In this paper, we give an affirmative answer to Problem~\ref{acyl_conj} 
in the latter case. 
Our main theorem is the following: 
\begin{theorem}\label{main}
Let $A_\Gamma$ be an irreducible Artin group associated with $\Gamma$. 
Suppose that $\Gamma$ is not a clique. 
Then $A_\Gamma$ is acylindrically hyperbolic.
\end{theorem}
From Theorem~\ref{main}, we find that 
many irreducible Artin groups of infinite type are acylindrically hyperbolic, e.g., 
the Artin group associated with the defining graph in Figure~\ref{example_fig}. 
We also apply Theorem~\ref{main} to Artin groups that are both of infinite type and of type FC.

\begin{corollary}
Let $A_\Gamma$ be an irreducible Artin group of infinite type
and of type FC.
Then $A_\Gamma$ is acylindrically hyperbolic.
\end{corollary}
\noindent
Note that, an Artin group of type FC is of spherical type if the defining graph is a clique.

Before we give other applications, we define the maximal clique factor of $\Gamma$.
Let $\Gamma$ be a simple graph that is not a clique.
If a factor of a join decomposition of $\Gamma$ is a clique,
then it is called a \textit{clique factor}.
If $\Gamma$ has a clique factor, then it has a unique maximal one,
which is called the \textit{maximal clique factor} and usually denoted by $\Gamma_0$ in this paper.
See Figures~\ref{example_fig} and \ref{example_fig2}.
For convenience, if $\Gamma$ has no nontrivial clique factor,
we set $\Gamma_0=\emptyset$ and refer to it as the maximal clique factor, as well.
See Section~\ref{join} for details related to join decompositions, and Figures~\ref{new_ex_fig} and \ref{new_ex_fig2} for other examples.
\begin{figure}
\begin{center}
\includegraphics[width=12cm,pagebox=cropbox,clip]{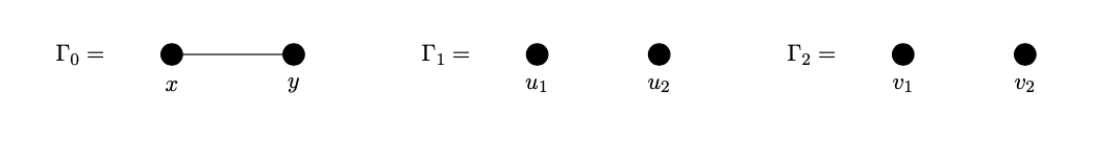} 
\caption{$\Gamma$ in Figure~\ref{example_fig} has a join decomposition $\Gamma=\Gamma_0\ast \Gamma_1\ast \Gamma_2$, and $\Gamma_0$ is the maximal clique factor. Later we denote $\Gamma_{\ast}=\Gamma_1\ast\Gamma_2$. See Section~\ref{join}.}\label{example_fig2}
\end{center}
\end{figure}

We apply Theorem~\ref{main} to the so-called center conjecture, which states that 
the center of any irreducible Artin group of infinite type is trivial.  
\begin{corollary}\label{main2}
Let $A_\Gamma$ be an irreducible Artin group associated with $\Gamma$ that is not a clique. 
Then the center $Z(A_\Gamma)$ is finite.  
Moreover, $Z(A_\Gamma)$ is contained in $Z(A_{\Gamma_0})$.
In particular, if $Z(A_{\Gamma_0})$ is trivial or torsion-free, then $Z(A_\Gamma)$ is trivial. 
Here, $\Gamma_0$ is the maximal clique factor of $\Gamma$. 
\end{corollary}
\noindent A result related to the last part of Corollary~\ref{main2} can be found in \cite{JM}.  
See Remark~\ref{center_rem}.

We also obtain a corollary concerning direct products. 
\begin{corollary}\label{main3}
Let $A_\Gamma$ be an irreducible Artin group associated with $\Gamma$ that is not a clique. 
Then $A_\Gamma$ is virtually directly indecomposable;
that is, it cannot be decomposed as a direct product 
of two infinite subgroups.
Moreover, if $A_{\Gamma_0}$ contains no nontrivial finite subgroup that is a factor of a direct decomposition of $A_{\Gamma_0}$, 
then $A_\Gamma$ is directly indecomposable. 
Here, $\Gamma_0$ is the maximal clique factor of $\Gamma$. 
\end{corollary}

\begin{example}\label{example2,3}
Let $A_\Gamma$ be an irreducible Artin group associated 
with $\Gamma$ that is not a clique. 
Suppose that the maximal clique factor $\Gamma_0$ consists of three vertices or fewer. 
See Figures~\ref{example_fig} and \ref{example_fig2}. 
Then $Z(A_{\Gamma_0})$ is known to be trivial or torsion-free, 
and thus $Z(A_\Gamma)$ is trivial by Corollary~\ref{main2}. 
Also, $A_{\Gamma_0}$ is known to be torsion-free, 
and thus $A_\Gamma$ is directly indecomposable by Corollary~\ref{main3}. 
\end{example}

\begin{remark}
Goldsborough and Vaskou, who are the authors of \cite{GV}, pointed to us out that the results in \cite{GV} and Theorem~\ref{main} imply acylindrical hyperbolicity of random Artin groups.
The following is based on their remark.
By referring to \cite{GV} for some definitions,
we can state the resulting implication precisely:
an Artin group picked at random relatively to
any non-decreasing divergent function $f:\mathbb{N}\to\mathbb{N}$ is
asymptotically almost surely acylindrically hyperbolic.
Indeed the case of $f(N)\succ N^{\frac{3}{2}}$ and
the case of $f(N)\prec N^{1-\alpha}$ for some $0<\alpha<1$
are covered in Corollaries 3.4 and 4.3 in \cite{GV}, respectively.
The missing case was unknown (see Question 1.6 in \cite{GV}),
but now it is resolved as follows:
we consider the case of $f(N)\prec N^2$, which contains the missing case.
By Lemma 2.7, Theorem 2.9 (2) and Lemma 2.10 in \cite{GV},
an Artin group picked at random relatively to $f$
asymptotically almost surely is irreducible and not free of infinity.
Thus, it is acylindrically hyperbolic by Theorem~\ref{main}.
\end{remark}

In this paper, for the proof of Theorem~\ref{main}, 
we develop arguments from \cite{Charney} 
and \cite{KO2}. 
Charney and Morris-Wright \cite{Charney}  established the acylindrical hyperbolicity of Artin groups of infinite type associated with graphs that are not joins, by studying clique-cube complexes, 
which are CAT(0) cube complexes, and the isometric actions on them. 
They constructed a WPD contracting element of such an Artin group with respect to the isometric action on the clique-cube complex. 
We note that the existence of a WPD contracting element is a sufficient condition for acylindrical hyperbolicity \cite{BBF}.
The authors \cite{KO2} developed this approach and treated Artin groups associated with graphs that are not cones. 
To treat 
the case of Artin groups associated with graphs that are not cliques, 
we need additional development and more refined arguments. 
Indeed, we introduce the reduced clique-cube complex. 
It is a convex subcomplex of the clique-cube complex, 
and works more appropriately in our case than the clique-cube complex. 
Together with careful analysis of the local geometry of the cube complexes, 
we 
construct a candidate of a WPD contracting element 
with respect to the action on the reduced clique-cube complex.
We also apply some algebraic results on standard parabolic subgroups of Artin groups to show that the candidate is actually a WPD contracting element.

The remainder of this paper is organized as follows. 
Section~\ref{preliminary_sec} contains some preliminaries. 
Section~\ref{recall} is for preliminaries regarding acylindrically hyperbolic groups, WPD contracting elements, 
and CAT(0) cube complexes. 
Section~\ref{defining_graph} presents preliminaries on defining graphs of Artin groups. 
Section~\ref{artin_group} presents facts on standard parabolic subgroups of Artin groups and gives a lemma on intersections of standard parabolic subgroups and their conjugates in Artin groups. 
Section~\ref{abs_clique} treats clique-cube complexes. 
Section~\ref{join} deals with the definition of joins of graphs and some related properties.
In Section~\ref{matrix}, we prepare some operations on matrices over a group, which are used when we construct a WPD contracting element in Section~\ref{main_proof}.
In Section~\ref{clique}, we introduce reduced clique-cube complexes 
and study the local geometry. 
In Section~\ref{main_proof}, we state Theorem~\ref{main-4} which contains Theorem~\ref{main} and adds information, and prove it. 
Our main task is to construct a candidate WPD contracting element
and show that it is really a WPD contracting element.
We also prove Corollaries~\ref{main2} and \ref{main3}. 
In Section~\ref{example_sec}, we treat an example to aid understanding of Section~\ref{main_proof}.

\section{preliminaries}\label{preliminary_sec}
\subsection{Acylindrical hyperbolicity, WPD contracting elements, and CAT(0) cube complexes}\label{recall}
We collect some definitions and properties related to acylindrical hyperbolicity, 
WPD contracting elements, and CAT(0) cube complexes for later use. 
See \cite{Genevois} and the references therein for details. 

First, we recall the definition of acylindrical hyperbolicity (see \cite{Osin}). 
\begin{definition}\label{acylindrically hyperbolic}
A group $G$ is \textit{acylindrically hyperbolic} if it admits an isometric action 
on a hyperbolic space $Y$ that is \textit{non-elementary}  (i.e., with an infinite limit set) 
and \textit{acylindrical} (i.e., for every $D\ge 0$, there exist some $R, N \ge 0$ such that, 
for all $y_1, y_2\in Y$, $d_Y(y_1, y_2)\ge R$ implies 
$\#\{g\in G |\ d_Y(y_1, g(y_1)), d_Y(y_2, g(y_2)) \le D\}\le N$).
\end{definition}

Next, we recall the definition of a WPD contracting element. 
\begin{definition}\label{WPD}
Let a group $G$ act isometrically on a metric space $X$. For $\gamma\in G$, 
we say that: 
\begin{itemize}
\item $\gamma$ is \textit{weakly properly discontinuous} (\textit{WPD}) if, for every $D \ge 0$ and $x \in X$, there exists some $M\ge 1$ 
such that 
$$\#\{g\in G |\ d_X(x, g(x)), d_X(\gamma^M(x), g\gamma^M(x)) \le D\}<\infty;$$ 
\item $\gamma$ is \textit{contracting} if $\gamma$ is \textit{loxodromic}, that is, 
there exists $x_0\in X$ such that the map $\mathbb Z \to X; n \mapsto \gamma^n(x_0)$ is a quasi-isometry onto the image 
$\gamma^\mathbb{Z}x_0:=\{\gamma^n(x_0) | n \in \mathbb Z\}$, 
and $\gamma^\mathbb{Z}x_0$ is \textit{contracting}; that is, there exists $B\ge 0$ such that 
the diameter of the nearest-point projection of any ball that is disjoint 
from $\gamma^\mathbb{Z}x_0$ onto $\gamma^\mathbb{Z}x_0$ is bounded by $B$. 
\end{itemize}
\end{definition}

The following is a consequence of \cite{BBF}. 
\begin{theorem}\label{BBF}
Let a group $G$ act isometrically on a geodesic metric space $X$. 
Suppose that $G$ is not virtually cyclic. 
If there exists a WPD contracting element $\gamma\in G$, then $G$ is acylindrically hyperbolic. 
\end{theorem}

CAT(0) cube complexes are regarded as generalized trees in higher dimensions. 
The following is a precise definition (see \cite[p.111]{BH}). 

\begin{definition}\label{CAT(0)}
A \textit{cube complex} is a $CW$ complex constructed by gluing together cubes of 
arbitrary (finite) dimension by isometries along their faces. 
Furthermore, the cube complex is \textit{nonpositively curved} if the link of any of its vertices is 
a \textit{simplicial flag} complex (i.e., $n+1$ vertices span an $n$-simplex if and only if they are pairwise adjacent), 
and \textit{CAT(0)} if it is nonpositively curved and simply connected. 
\end{definition}

\begin{definition}
Let $X$ be a CAT(0) cube complex. 
We define an equivalence relation on the edges of $X$ as 
the transitive closure of the relation identifying two parallel edges of a square. 
For an equivalence class, a \textit{hyperplane} is defined as 
the union of the midcubes transverse to the edges belonging to the equivalence class. 
Then, for any edge belonging to the equivalence class, 
the hyperplane is said to be \textit{dual to} the edge.

For a hyperplane $J$, 
we denote 
the union of the cubes intersecting $J$ as $N(J)$; that is, 
the smallest subcomplex of $X$ containing $J$. 
We denote 
the union of the cubes not intersecting $J$ as $X\!\setminus\!\!\setminus J$;
that is, the largest subcomplex of $X$ not intersecting $J$. 
\end{definition}

See \cite{Sageev} for the following. 
\begin{theorem}\label{separate}
Let $X$ be a CAT(0) cube complex and $J$ a hyperplane. 
Then, $X\!\setminus\!\!\setminus J$ has exactly two connected components. 
\end{theorem}
\noindent
The two connected components of $X\!\setminus\!\!\setminus J$ 
are often denoted by $J^+$ and $J^-$, respectively. 

The following is adapted from \cite[Definition 2.7]{KO2}. 
\begin{definition}\label{separatingdef}
Let $X$ be a CAT(0) cube complex. 
For two vertices $x$ and $x'$ in $X$, we call a sequence of hyperplanes 
$P_1,\ldots,P_M$ a \textit{sequence of separating hyperplanes from $x$ to $x'$} 
if the sequence satisfies 
$$x \in P_1^-, P_1^+\supsetneq P_2^+\supsetneq\cdots 
\supsetneq P_{M-1}^+\supsetneq P_M^+\ni x'$$ 
for some connected components $P_i^+$ of $X\!\setminus\!\!\setminus P_i$ 
for all $i\in \{1,\ldots,M\}$. 

For two hyperplanes $J$ and $J'$ in $X$, we call a sequence of hyperplanes 
$P_1,\ldots,P_M$ a \textit{sequence of separating hyperplanes from $J$ to $J'$} 
if the sequence satisfies 
$$J^+ \supsetneq P_1^+\supsetneq P_2^+\supsetneq\cdots 
\supsetneq P_{M-1}^+\supsetneq P_M^+ \supsetneq J'^+$$ 
for some connected components $J^+$ of $X\!\setminus\!\!\setminus J$, 
$J'^+$ of $X\!\setminus\!\!\setminus J'$, 
and $P_i^+$ of $X\!\setminus\!\!\setminus P_i$ 
for all $i\in \{1,\ldots,M\}$. 
\end{definition}
\begin{remark}\label{separatingrem}
When $P_1,\ldots,P_M$ is a sequence of separating hyperplanes from $J$ to $J'$, 
for two vertices $x\in J^-\cup N(J)$ and $x'\in J'^+\cup N(J')$, 
$P_1,\ldots,P_M$ is a sequence of separating hyperplanes from $x$ to $x'$. 
\end{remark}

The following is based on part of \cite[Theorem 3.3]{Genevois} 
and is used in the proof of Theorem~\ref{main}. 
\begin{theorem}\label{criterion}
Let a group $G$ act isometrically on a CAT(0) cube complex $X$. 
Then, $\gamma \in G$ is a WPD contracting element if there exist 
two hyperplanes $J$ and $J'$ satisfying the following: 
\begin{enumerate}
\item[(i)] $J$ and $J'$ are strongly separated; that is, 
no hyperplane intersects both $J$ and $J'$; 
\item[(ii)] $\gamma$ skewers $J$ and $J'$; that is, we have connected components $J^+$ of $X\!\setminus\!\!\setminus J$ and 
$J'^+$ of $X\!\setminus\!\!\setminus J'$ such that 
$\gamma^n (J^+) \subsetneq J'^+ \subsetneq J^+$ for some $n\in\mathbb N$;
\item[(iii)] $\mathrm{stab}(J) \cap \mathrm{stab} (J')$ is finite, where 
$\mathrm{stab}(J)=\{g\in G\  | \ g(J)=J\}$ and $\mathrm{stab} (J')=\{g\in G\  | \ g(J')=J'\}$. 
\end{enumerate}
\end{theorem}

\subsection{Defining graphs of Artin groups}\label{defining_graph}
We present a precise description of 
the defining graph of an Artin group and introduce some related graphs. 

Let $V$ be a finite set. 
Denote the diagonal set as $\mathrm{diag}(V\times V):=\{(v,w)\in V\times V |\ v=w\}$. 
We consider the involution on the off-diagonal set 
$$\iota:V\times V\setminus\mathrm{diag}(V\times V)\ni (v,w)\mapsto (w,v)\in V\times V\setminus\mathrm{diag}(V\times V).$$ 
Any $e\in V\times V\setminus\mathrm{diag}(V\times V)$ is often presented as $(s_e,t_e)$. 
Then, for any $e\in V\times V\setminus\mathrm{diag}(V\times V)$, 
we have $s_{\iota(e)}=t_e$ and $t_{\iota(e)}=s_e$. 
We take a symmetric map 
$$\tilde{\mu}:V\times V\setminus\mathrm{diag}(V\times V) \to \mathbb{Z}_{\ge 2}\cup \{\infty\}.$$ 
Here, `\textit{symmetric}' means that $\tilde{\mu}\circ\iota=\tilde{\mu}$ is satisfied. 
Then, we have 
$$V\times V\setminus\mathrm{diag}(V\times V)=\bigsqcup_{m\in \mathbb{Z}_{\ge 2}\cup \{\infty\}}\tilde{\mu}^{-1}(m).$$
We now have a finite simple labeled graph $\Gamma$ 
with the vertex set $V(\Gamma)=V$, 
the edge set $E(\Gamma)=\bigsqcup_{m\in \mathbb{Z}_\ge 2}\tilde{\mu}^{-1}(m)$, 
and the labeling $\mu:=\tilde{\mu}|_{E(\Gamma)}$. 
The Artin group $A_\Gamma$ associated with $\Gamma$ 
is defined by the presentation (\ref{pres_standard_eq}),
and $\Gamma$ is called the \textit{defining graph} of $A_\Gamma$. 

For convenience, we define 
two other finite simple graphs, $\Gamma^c$ and $\Gamma^t$, as follows. 
$\Gamma^c$ is the finite simple graph 
with the vertex set $V(\Gamma^c)=V$ 
and the edge set $E(\Gamma^c)=\tilde{\mu}^{-1}(\infty)$. 
This is called the \textit{complement graph} of $\Gamma$. 
The \textit{Coxeter graph} $\Gamma^t$ is the finite simple graph with the vertex set $V(\Gamma^t)=V$ 
and the edge set $E(\Gamma^t)=\bigsqcup_{m\in \mathbb{Z}_{\ge 3}\cup\{\infty\}}\tilde{\mu}^{-1}(m)$. 

\subsection{Properties of Artin groups}\label{artin_group}
We summarize some properties of Artin groups. 
Let $A_\Gamma$ be an Artin group associated with a defining graph $\Gamma$ with the vertex set $V(\Gamma)=V$ and the edge set $E(\Gamma)=E$, 
as in Section~\ref{defining_graph}. 
Lemma~\ref{Godelle} is essentially used in the proof of Theorem~\ref{main}. 

By the work of van der Lek (see \cite{Paris}, \cite{van}), the following are well known and often used without mention in this paper. 
\begin{lemma}\label{van}
\begin{enumerate}
\item For any subset $U\subset V$, 
the subgroup of $A_\Gamma$ generated by $U$, called a standard parabolic subgroup, 
is itself an Artin group associated 
with the full subgraph of $\Gamma$ spanned by $U$. 
(We denote this subgroup by $A_U$. 
When $U$ is empty, we define $A_\emptyset =\{1\}$.)
\item For any subsets $U, U'\subset V$, $A_{U}\cap A_{U'}=A_{U\cap U'}$. 
\end{enumerate}
\end{lemma}

The following is a consequence of a nontrivial result by Godelle \cite{Godelle}. 
\begin{lemma}\label{Godelle}
Take an element $v\in V$ 
and a subset $U\subset V$. Suppose that $v\not\in U$. 
Set $W:=\{w\in U\mid \mu(v,w)=2\}$.
Then $A_{U}\cap vA_{U}v^{-1}=A_{W}$. 
In particular, for any $u\in U\setminus W$, 
$A_{U}\cap vA_{U}v^{-1}\subset A_{U\setminus\{u\}}$. 
\end{lemma}
\begin{proof}
Clearly we have $A_{W}\subset A_{U}\cap vA_{U}v^{-1}$.
In terms of \cite{Godelle}, $v$ is $(U,U)$-reduced in the Coxeter group $W_{\Gamma}$.
Applying \cite[Proposition 3.3]{Godelle} to our setting,
we have some $U'\subset U$ such that
$vA_{U}v^{-1}\cap A_{U}= A_{U'}$.
Hence $W\subset U'$.
Assume that $W\subsetneq U'$.
Take $u'\in U'\setminus W$.
Then $u'\in A_{U'}=vA_{U}v^{-1}\cap A_{U}$
and thus $v^{-1}u'v\in A_{U}$.
By $v^{-1}u'v\in A_{\{v,u'\}}$ and $\{v,u'\}\cap U=\{u'\}$,
we have $v^{-1}u'v\in A_{\{v,u'\}}\cap A_{U}=A_{\{u'\}}$.
Therefore, we have $m\in \mathbb Z$ such that $v^{-1}u'v=(u')^m$.
Since the surjection $A_{\{v,u'\}}\to A_{\{u'\}}$ defined by $v, u'\mapsto u'$ gives
$v^{-1}u'v\mapsto u'$ and $(u')^m\mapsto (u')^m$, we see $u'={(u')}^m$ and $m=1$.
Hence we have $v^{-1}u'v=u'$ in $A_\Gamma$.
This means $\mu(v,u')=2$, contradicting $u'\notin W$.
Hence $U'=W$. Therefore, $vA_{U}v^{-1}\cap A_{U}=A_{W}$.
\end{proof}

\subsection{Clique-cube complexes and actions on them}\label{abs_clique}
We consider clique-cube complexes and the actions on them 
following \cite{Charney}. 
Let $A_\Gamma$ be an Artin group associated with a defining graph $\Gamma$ with the vertex set $V(\Gamma)=V$ and the edge set $E(\Gamma)=E$, 
as in Section~\ref{defining_graph}. 
We say that $U$ \textit{spans a clique} in $\Gamma$ if any two elements of $U$ are joined by an edge in $\Gamma$. 

\begin{definition}[{\cite[Definition 2.1]{Charney}}]\label{clique cube}
Consider the set 
$$\Delta_\Gamma = \{U \subset V |\ U\text{ spans a clique in }\Gamma,\text{ or }U = \emptyset\}.$$ 
The \textit{clique-cube complex} $C_\Gamma$ is the cube complex 
whose vertices (i.e., $0$-dimensional cubes) are cosets $gA_U$ ($g\in A_\Gamma, U\in \Delta_\Gamma$), 
where two vertices $gA_U$ and $hA_{U'}$ 
are joined by an edge (i.e., a $1$-dimensional cube) in $C_\Gamma$ 
if $gA_U\subset hA_{U'}$ and $U$ and $U'$ differ by a single generator. 
Note that, in this case, 
we can always replace $h$ by $g$ so that $hA_{U'}=gA_{U'}$. 
More generally, two vertices $gA_U$ and $gA_{U'}$ with $gA_U\subset gA_{U'}$ 
span a $\#(U'\setminus U)$-dimensional cube $[gA_U , gA_{U'}]$ in $C_\Gamma$. 
\end{definition}
\noindent See Figure~\ref{reduced-complex_fig}.

\begin{figure}
\begin{center}
\includegraphics[width=10cm,pagebox=cropbox,clip]{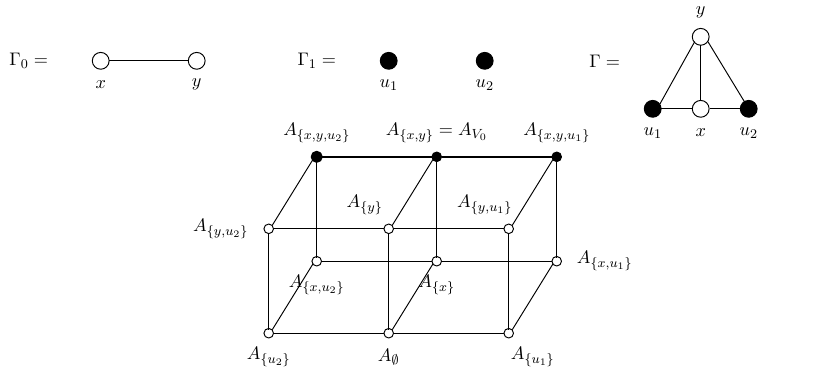} 
\caption{For the above graph $\Gamma$, the below cube complex is a part of $C_{\Gamma}$ that is a fundamental domain of $C_{\Gamma}$ with respect to the action of $A_{\Gamma}$. 
We define a subcomplex $C_{\Gamma,\Gamma_0}$ of $C_{\Gamma}$ in Section~\ref{clique}.
The subcomplex consisting of colored vertices and thick edges is a part of $C_{\Gamma,\Gamma_0}$, 
which is a fundamental domain of $C_{\Gamma,\Gamma_0}$ with respect to the action of $A_{\Gamma}$.}
\label{reduced-complex_fig}
\end{center}
\end{figure}

We endow a metric on $C_{\Gamma}$ in the standard way for a cube complex.
That is, we identify all cubes in $C_{\Gamma}$ with Euclidean unit cubes and give $C_{\Gamma}$ the path metric.
Note that $C_{\Gamma}$ is connected.

\begin{theorem}[{\cite[Theorem 2.2]{Charney}}]\label{C_Gamma_thm}
The clique-cube complex $C_\Gamma$ is a finite dimensional CAT(0) cube complex. In particular, $C_{\Gamma}$ is a complete CAT(0) space. 
\end{theorem} 
\noindent Note that completeness follows from finite dimensionality of $C_{\Gamma}$ (\cite[Theorem 7.19]{BH}).

The group $A_\Gamma$ acts on the clique-cube complex $C_\Gamma$ by left multiplication, 
$h\cdot gA_U = (hg)A_U$. 
This action preserves the cubical structure and is isometric. 
The action is also cocompact with a fundamental domain $\bigcup_{U\in \Delta_\Gamma}[A_\emptyset, A_U]$,
where $[A_\emptyset, A_U]$ is a $\# U$-dimensional cube spanned by the two vertices $A_\emptyset$ and $A_U$ in $C_\Gamma$. 
See Figure~\ref{reduced-complex_fig}.
However, the action is not proper. 
In fact, the stabilizer of a vertex $gA_U$ is the conjugate subgroup 
$gA_U g^{-1}$, so all vertices other than those in the orbit of $A_\emptyset$ have infinite stabilizers. 
We also note that $C_\Gamma$ is not a proper metric space because it contains vertices of infinite valence. 
Additionally, $C_\Gamma$ has infinite diameter if and only if $\Gamma$ itself is not a clique. 

\begin{remark}[{\cite[Section 2]{Charney}}]\label{label}
Each edge in $C_\Gamma$ can be labeled with a generator in $V$. 
Indeed, the edge between $gA_U$ and $gA_{U\sqcup \{v\}}$ is labeled by $v$. 
Any two parallel edges in a cube have the same label, 
so we can also label the hyperplane dual to such an edge by $v$ 
and say that such a hyperplane is \textit{of $v$-type}. 
Every hyperplane of $v$-type is a translation of the hyperplane 
dual to the edge between $A_\emptyset$ and $A_{\{v\}}$. 
If a hyperplane of $v$-type crosses another hyperplane of $v'$-type, then $(v,v')\in E$. 
In particular, two different hyperplanes of the same type do not cross each other.
\end{remark}

\subsection{Joins of graphs}\label{join}
\begin{definition}
Let $\Lambda\neq \emptyset$ be an index set. 
The \textit{join} $\ast_{\alpha\in \Lambda}\Gamma_\alpha$ 
of simple graphs $\Gamma_\alpha$ ($\alpha\in \Lambda$) is defined as a simple graph with 
the vertex set $$V(\ast_{\alpha\in \Lambda} \Gamma_\alpha):=\bigsqcup_{\alpha\in \Lambda}  V(\Gamma_\alpha)$$ 
and the edge set $$E(\ast_{\alpha\in \Lambda} \Gamma_\alpha):=\bigsqcup_{\alpha\in \Lambda} E(\Gamma_\alpha)\sqcup 
\bigsqcup_{\alpha,\beta\in \Lambda, \alpha\neq \beta}\{(v_\alpha, v_\beta)\mid v_\alpha\in V(\Gamma_\alpha), v_\beta\in V(\Gamma_\beta)\}.$$

A simple graph $\Gamma$ is said to be \textit{decomposable as a join} 
if there exists an index set $\Lambda$ with $\# \Lambda\ge 2$ and 
subgraphs $\Gamma_\alpha$ ($\alpha\in \Lambda$) of $\Gamma$ 
such that $\Gamma=\ast_{\alpha \in \Lambda}\Gamma_\alpha$. 
This is called a \textit{join decomposition} of $\Gamma$ into factors $\Gamma_\alpha$ ($\alpha\in \Lambda$). 
$\Gamma$ is said to be \textit{indecomposable as a join} 
if it is not decomposable as a join. 

\end{definition}
\begin{remark}
A simple graph $\Gamma$ is indecomposable as a join if and only if its complement graph $\Gamma^c$ is connected.
\end{remark}

The following is a well-known fact (See \cite[Lemma 3.4]{KO2}). 
\begin{lemma}\label{join-decomposition_lem}
Let $\Gamma$ be a simple graph. 
Suppose that $\Gamma$ is decomposable as a join. 
Then, $\Gamma$ has a unique join decomposition into indecomposable factors 
$$\Gamma=\ast_{\alpha\in \Lambda}\Gamma_\alpha,$$
and the decomposition of $\Gamma^c$ into connected components is given by 
$$\Gamma^c=\bigsqcup_{\alpha\in \Lambda}(\Gamma_\alpha)^c.$$
\end{lemma}
\noindent For convenience, even if $\Gamma$ is indecomposable as a join,
we take $\Lambda$ consisting of a single index $\alpha$,
define $\Gamma_{\alpha}=\Gamma$ and regard
$\Gamma=\ast_{\alpha\in \Lambda}\Gamma_{\alpha}$ as the unique join decomposition of $\Gamma$
(by a single indecomposable factor).

\begin{figure}
\begin{center}
\includegraphics[width=8cm,pagebox=cropbox,clip]{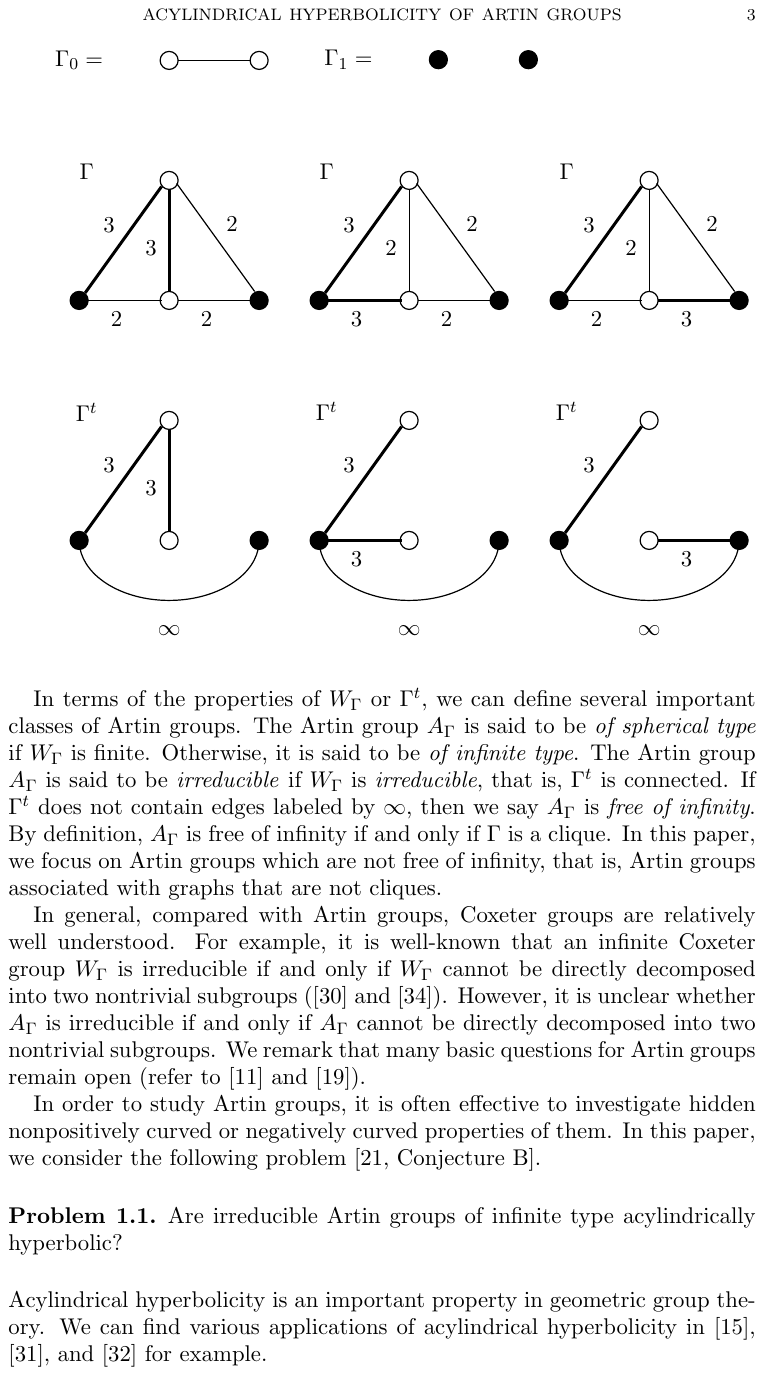} 
\caption{The defining graphs $\Gamma$ 
and the associated Coxeter graphs $\Gamma^t$. 
Each $\Gamma$ has a join decomposition $\Gamma=\Gamma_0\ast \Gamma_1$, where $\Gamma_0$ is the maximal clique factor and $\Gamma_1$ 
is an indecomposable factor with two vertices.}\label{new_ex_fig}
\end{center}
\end{figure}

\begin{figure}
\begin{center}
\includegraphics[width=10cm,pagebox=cropbox,clip]{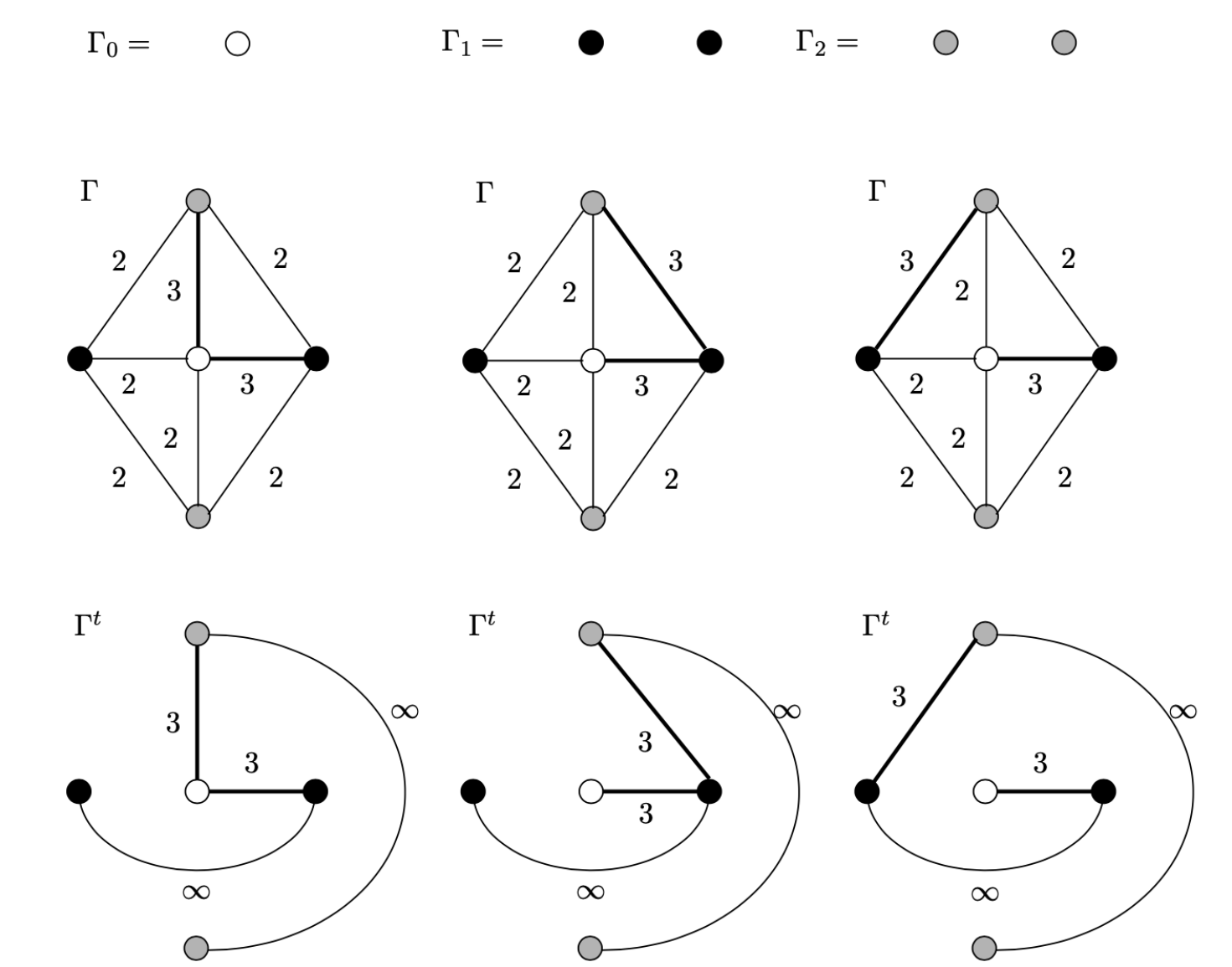} 
\caption{The defining graphs $\Gamma$ 
and the associated Coxeter graphs $\Gamma^t$. 
Each $\Gamma$ has a join decomposition $\Gamma=\Gamma_0\ast \Gamma_1\ast \Gamma_2$, where $\Gamma_0$ is the maximal clique factor and 
each of $\Gamma_1$ and 
$\Gamma_2$ is an indecomposable factor with two vertices.}\label{new_ex_fig2}
\end{center}
\end{figure}

\begin{definition}\label{decompositon}
Let $\Gamma$ be a simple graph. 
For the unique join decomposition into indecomposable factors 
$\Gamma=\ast_{\alpha\in \Lambda}\Gamma_\alpha$, 
we set $\Lambda_0:=\{\alpha\in \Lambda\mid \ \#\Gamma_\alpha=1\}$ and 
$\Gamma_0:=\ast_{\alpha\in \Lambda_0}\Gamma_\alpha$. 
We call $\Gamma_0$ the \textit{maximal clique factor} of $\Gamma$. 
Also we set $\Lambda_\ast:=\Lambda\setminus\Lambda_0$ and 
$\Gamma_\ast:=\ast_{\alpha\in \Lambda_\ast}\Gamma_\alpha$.  
We call $\Gamma$ a \textit{cone} if $\Lambda_0\neq \emptyset$.
\end{definition}
\noindent See Figures~\ref{example_fig}, \ref{example_fig2}, \ref{new_ex_fig} and \ref{new_ex_fig2}. 

We have $\Gamma=\Gamma_0\ast \Gamma_\ast$. 
Clearly, $\Gamma$ is a clique if and only if $\Lambda=\Lambda_0$. 

\subsection{Notation on matrices}\label{matrix}
For natural numbers $k$, $m$, and a set $X$,
let $\mathcal{M}(k,m; X)$ be the set of all $k\times m$ matrices whose entries are in $X$.
For $i\in \{1,\ldots, k\}$ and $l\in \{1,\ldots, m\}$,
let $M[i,l]$ be the $(i,l)$-component of $M$.
Let $M[i,\ ]$ be the $i$-th row and 
$M[\ ,l]$ be the $l$-th column of the matrix $M$.
For $l_1, l_2\in \{1,\ldots, m\}$ with $l_1\leq l_2$, 
we define $M[\ ,l_1:l_2]\in \mathcal{M}(k, l_2-l_1+1; X)$
by 
\begin{align}
\left(M[\ ,l_1:l_2]\right)[\ ,l]=M[\ ,l_1-1+l]\quad(1\leq l\leq l_2-l_1+1).
\end{align}
For two matrices $M_j\in \mathcal{M}(k, m_j;X)$ $(j=1,2)$ with the same number of rows,
we define $cbind(M_1,M_2)\in \mathcal{M}(k, m_1+m_2;X)$ as a matrix constructed from $M_1,M_2$ by binding the columns.
That is,  
\begin{align}\label{cbind-def}
cbind(M_1,M_2)[\ ,l]=
\begin{cases}
M_1[\ ,l]\quad(1\leq l\leq m_1)\\
M_2[\ ,l-m_1]\quad(m_1+1\leq l\leq m_1+m_2).
\end{cases}
\end{align}
For matrices $M_j\in \mathcal{M}(k, m_j;X)$ $(1\leq j\leq r)$ with the same number of rows,
we define $cbind(M_1,\ldots,M_r)\in \mathcal{M}(k, m_1+\cdots+m_r;X)$ inductively by 
\begin{align}
cbind(M_1,\ldots,M_r)=cbind\left(cbind(M_1,\ldots,M_{r-1}),M_r\right).
\end{align}

Let $A$ be a group and let $M\in \mathcal{M}(k, m;A)$. 
For any subset $B\subset A$, which is typically a subgroup or a generating set, we naturally regard $\mathcal{M}(k,m;B)\subset \mathcal{M}(k,m;A)$. 
For two matrices $M_1, M_2\in \mathcal{M}(k, m;A)$ of the same size, 
$M_1\odot M_2\in \mathcal{M}(k, m;A)$ denotes the Hadamard product of $M_1$ and $M_2$.
That is,   
\begin{equation}\label{odot_map}
(M_1\odot M_2)[i,l]=M_1[i,l]\cdot M_2[i,l]\in A.
\end{equation}
We define a map $prod: \mathcal{M}(k, m;A)\to A$ inductively as follows.
When $m=1$, let
\begin{equation}\label{prod1_map}
prod(M)=M[1,1]\cdot M[2,1]\cdot \cdots \cdot M[k,1]\in A.
\end{equation}
If $prod$ is defined on $\mathcal{M}(k, m-1;A)$, we define $prod: \mathcal{M}(k, m;A)\to A$ by
\begin{equation}\label{prod2_map}
prod(M)=prod(M[\ ,1:m-1])\cdot prod(M[\ ,m])\in A.
\end{equation}

\section{Reduced clique-cube complexes and actions on them}\label{clique}

In this section, we introduce and study reduced clique-cube complexes. 
Let $A_\Gamma$ be an Artin group associated with a defining graph $\Gamma$ with the vertex set $V(\Gamma)=V$ and the edge set $E(\Gamma)=E$,
as in Section~\ref{defining_graph}. 
Let $\Gamma_0$ be the maximal clique factor of $\Gamma$.
We consider the join decomposition $\Gamma=\Gamma_0\ast \Gamma_{\ast}$.
We set $V_0=V(\Gamma_0)$.

\begin{definition}\label{reduced_def}
Let
$$\Delta_{\Gamma,\Gamma_0}=\{U\subset V\mid U\in \Delta_{\Gamma},\ U\supset V_0\}.$$
Let $C_{\Gamma,\Gamma_0}$ be the subcomplex of $C_\Gamma$ 
spanned by vertices $gA_U$ with $g\in A_\Gamma$ and 
$U\in \Delta_{\Gamma,\Gamma_0}$.
We call $C_{\Gamma,\Gamma_0}$ the \textit{reduced clique-cube complex}.
\end{definition}
\noindent See Figure~\ref{reduced-complex_fig}.

When $\Gamma$ is not a cone; that is, $\Gamma=\Gamma_\ast$, we have $C_{\Gamma,\Gamma_0}=C_{\Gamma}$.
When $\Gamma$ is a clique; that is, $\Gamma=\Gamma_0$, we have $C_{\Gamma,\Gamma_0}=C_{\Gamma,\Gamma}=\{A_{\Gamma}\}$. 
In this paper, we are mainly interested in the case where $\Gamma$ is not a clique.

We note that the action of $A_{\Gamma}$ on $C_{\Gamma}$ preserves $C_{\Gamma,\Gamma_0}$.
We consider the induced action of $A_{\Gamma}$ on $C_{\Gamma,\Gamma_0}$.
Then $\bigcup_{U\in \Delta_{\Gamma,\Gamma_0}}[A_{V_0}, A_U]$ is a fundamental domain of $C_{\Gamma,\Gamma_0}$
with respect to the action of $A_{\Gamma}$.

\begin{lemma}\label{connected_lem}
$C_{\Gamma,\Gamma_0}$ is a connected subcomplex of $C_{\Gamma}$.  
\end{lemma}

\begin{proof}
If $\Gamma$ is a clique, then $C_{\Gamma,\Gamma_0}=\{A_{\Gamma}\}$ is clearly connected. 
Suppose that $\Gamma$ is not a clique.
Let $U\in \Delta_{\Gamma,\Gamma_0}$.
Let $g\in A_{\Gamma}$. 
First, we show that there exists a path in $C_{\Gamma,\Gamma_0}$ from $A_{V_0}$ to $gA_{V_0}$.
Suppose that $g\notin A_{V_0}$; otherwise, we would have $A_{V_0}=gA_{V_0}$.
Let $g=v_1^{\varepsilon_1}\cdots v_n^{\varepsilon_n}$, where $n\in \mathbb{N}$, $v_i\in V$, $\varepsilon_i\in \{1,-1\}$ for every $1\leq i\leq n$.
By induction on $n$, we suppose that there exists a path in $C_{\Gamma,\Gamma_0}$ from $A_{V_0}$ to $v_1^{\varepsilon_1}\cdots v_{n-1}^{\varepsilon_{n-1}}A_{V_0}$.
If $v_n\in V_0$, then 
$$v_1^{\varepsilon_1}\cdots v_{n-1}^{\varepsilon_{n-1}}A_{V_0}=v_1^{\varepsilon_1}\cdots v_{n}^{\varepsilon_{n}}A_{V_0},$$
and the induction proceeds.
If $v_n\notin V_0$, since $V_0\sqcup\{v_n\}\in \Delta_{\Gamma,\Gamma_0}$,
there exists a path
$$\left(v_1^{\varepsilon_1}\cdots v_{n-1}^{\varepsilon_{n-1}}A_{V_0},\ v_1^{\varepsilon_1}\cdots v_{n-1}^{\varepsilon_{n-1}}A_{V_0\sqcup\{v_n\}}=gA_{V_0\sqcup\{v_n\}},\ gA_{V_0}\right)$$
in $C_{\Gamma,\Gamma_0}$
from $v_1^{\varepsilon_1}\cdots v_{n-1}^{\varepsilon_{n-1}}A_{V_0}$ to $gA_{V_0}$.
By concatenating these two paths,
we obtain a path from $A_{V_0}$ to $gA_{V_0}$ in $C_{\Gamma,\Gamma_0}$.
Next, we show the existence of a path in $C_{\Gamma,\Gamma_0}$ from $gA_{V_0}$ to $gA_{U}$.
Let $U=V_0\sqcup \{u_1,\ldots, u_m\}$ where $u_i\neq u_j$ for $1\leq i\neq j\leq m$.
Then there exists a path
$$\left(gA_{V_0}, gA_{V_0\sqcup\{u_1\}}, gA_{V_0\sqcup\{u_1, u_2\}},\ldots, gA_{U}\right)$$
in $C_{\Gamma,\Gamma_0}$ from $gA_{V_0}$ to $gA_{U}$.
Therefore, it follows that $C_{\Gamma,\Gamma_0}$ is a connected subcomplex of $C_{\Gamma}$.
\end{proof}

We endow $C_{\Gamma,\Gamma_0}$ with a metric in the standard way for a cube complex.
That is, we identify all cubes in $C_{\Gamma,\Gamma_0}$ with Euclidean unit cubes and give $C_{\Gamma,\Gamma_0}$ the path metric.
Note that $C_{\Gamma,\Gamma_0}$ is connected.

\begin{definition}\label{convex_def}
Let $X$ be a geodesic space and $W$ be a subset of $X$.
$W$ is {\it convex} in $X$ if for every $w_1,w_2\in W$, every geodesic from $w_1$ to $w_2$ is contained in $W$.
\end{definition}

\begin{definition}\label{full_def}
Let $L$ be a simplicial complex and $L'$ be a subcomplex of $L$.
$L'$ is {\it full} if any simplex of $L$ whose vertices are in $L'$ is contained in $L'$.
\end{definition}

\begin{theorem}[see, e.g., \cite{SS}]\label{convex-subcomplex_thm}
Let $X$ be a finite dimensional CAT(0) cube complex and $W$ be a connected subcomplex.
Then $W$ is convex in $X$ if and only if for every vertex $v$ of $W$, the link $L'$ of $v$ in $W$ is a full subcomplex of the link $L$ of $v$ in $X$.
\end{theorem}

\begin{lemma}\label{convex_lem}
$C_{\Gamma,\Gamma_0}$ is convex in $C_{\Gamma}$.
In particular, $C_{\Gamma,\Gamma_0}$ is a finite dimensional CAT(0) cube complex and a complete CAT(0) space.
\end{lemma}

\begin{proof}
We show that $C_{\Gamma,\Gamma_0}$ is convex in $C_{\Gamma}$.
Let $gA_U$ be a vertex of $C_{\Gamma,\Gamma_0}$.
Let $L$ be the link of $gA_U$ in $C_{\Gamma}$ and let $L'$ be the link of $gA_U$ in $C_{\Gamma,\Gamma_0}$.
According to Theorem~\ref{convex-subcomplex_thm} and Lemma~\ref{connected_lem}, it is sufficient to show that $L'$ is a full subcomplex of $L$.
Let $\Delta$ be a $d$-dimensional simplex of $L$ whose vertices are in $L'$.
Let $C_{\Delta}$ be the $(d+1)$-dimensional cube of $C_{\Gamma}$ corresponding to $\Delta$.
Let $gA_{U_1},\ldots, gA_{U_{d+1}}$ be vertices of $C_{\Delta}$ that are connected to $gA_{U}$ by a $1$-dimensional cube of $C_{\Gamma}$.
Note that $C_{\Delta}=[gA_{\cap_{i=1}^{d+1} U_i}, gA_{\cup_{i=1}^{d+1} U_i}]$.
Since $gA_{U_i}\in C_{\Gamma,\Gamma_0}$, we have $U_i\supset V_0$ for every $i$.
Therefore, $\cap_{i=1}^{d+1} U_i\supset V_0$.
Since $C_{\Delta}$ is contained in $C_{\Gamma,\Gamma_0}$, $\Delta$ is contained in $L'$.
It follows that $L'$ is a full subcomplex of $L$.

The other assertions follow from the convexity of $C_{\Gamma,\Gamma_0}$ in $C_{\Gamma}$ and Theorem~\ref{C_Gamma_thm}.
\end{proof}

The following is a generalization of \cite[Lemma 3.2]{Charney}. 
\begin{lemma}\label{CM Lemma 3.2}
The action of $A_{\Gamma}$ on $C_{\Gamma,\Gamma_0}$ is a minimal action. 
That is, for every $x\in C_{\Gamma,\Gamma_0}$,
the convex hull of the orbit of $x$, $\mathrm{Hull}(orb(x))$, is $C_{\Gamma,\Gamma_0}$.
\end{lemma}
\begin{proof}
When $\Gamma$ is a clique, $C_{\Gamma,\Gamma_0}=\{A_{V_0}\}$. In this case, the lemma is trivial.
We now treat the case where $\Gamma$ is not a clique.
We fix a join decomposition 
$$\Gamma=\Gamma_0\ast \Gamma_1\ast\cdots \ast\Gamma_k,$$
where 
$\Gamma_0$ is the maximal clique factor of $\Gamma$ 
and $\Gamma_i$ are indecomposable factors with at least two vertices for $i\geq 1$.
We note that $\Gamma_1\ast\cdots \ast\Gamma_k\neq \emptyset$, since $\Gamma$ is not a clique.

Let $x\in C_{\Gamma,\Gamma_0}$. Let $C$ be a cube of $C_{\Gamma,\Gamma_0}$ with $x\in C$.
Without loss of generality, we set $C=[A_U, A_T]$ where $U\subset T$, $U,T\in \Delta_{\Gamma,\Gamma_0}$.
By induction on $\#(T\setminus U)$, we show that $\mathrm{Hull}(orb(x))=C_{\Gamma,\Gamma_0}$.

When $\#(T\setminus U)=0$, then $x=A_U$.
It is sufficient to show that $\mathrm{Hull}(orb(A_U))=\mathrm{Hull}(orb(A_{V_0}))$.
Indeed, if $\mathrm{Hull}(orb(A_U))=\mathrm{Hull}(orb(A_{V_0}))$ for every $U\in \Delta_{\Gamma,\Gamma_0}$, then $\mathrm{Hull}(orb(A_{V_0}))$ contains every vertex of $C_{\Gamma,\Gamma_0}$ and thus equals $C_{\Gamma,\Gamma_0}$.

First, we show that $\mathrm{Hull}(orb(A_U))\supset \mathrm{Hull}(orb(A_{V_0}))$.
If $V_0\neq U$, we take $u\in U\setminus V_0$.
There exists $i\in \{1,\ldots, k\}$ such that $u\in V(\Gamma_i)$.
Since $\Gamma_i$ is indecomposable, 
we have a vertex $s\neq u$ of $\Gamma_i$ such that $s$ and $u$ are not connected by an edge in $\Gamma_i$.
Note that $s\notin U$.
Let 
$$U'=\{v\in U\mid \text{ $v$ and $s$ are connected by an edge in $\Gamma$}\}.$$
We note that ${U'}$ does not contain $u$.
Hence we have $V_0\subset {U'}\subsetneq U$.
We also note that ${U'}\sqcup\{s\}$ spans a clique but $U\sqcup\{s\}$ does not.
Then, the sequence of vertices
$$A_U,\ A_{U'},\ A_{{U'}\sqcup\{s\}},\ sA_{U'},\ sA_U$$
forms a local geodesic, which is a geodesic, since $C_{\Gamma,\Gamma_0}$ is CAT(0) (Lemma~\ref{convex_lem}).
This geodesic is contained in $\mathrm{Hull}(orb(A_U))$, since the two endpoints are in $\mathrm{Hull}(orb(A_U))$.
Therefore, $A_{U'}\in \mathrm{Hull}(orb(A_U))$, and thus
$\mathrm{Hull}(orb(A_{U'}))\subset \mathrm{Hull}(orb(A_U))$.
If $U'=V_0$, then we have nothing further to show.
If ${U'}\supsetneq V_0$, we repeat the same argument until we obtain $\mathrm{Hull}(orb(A_U))\supset \mathrm{Hull}(orb(A_{V_0}))$.

Next, we show that $\mathrm{Hull}(orb(A_U))\subset \mathrm{Hull}(orb(A_{V_0}))$.
If $V_0\neq U$, we take $u\in U\setminus V_0$.
Then 
$$A_{U\setminus\{u\}},\ A_U,\ uA_{U\setminus \{u\}}$$
is a geodesic.
It follows that $\mathrm{Hull}(orb(A_U))\subset \mathrm{Hull}(orb(A_{U\setminus \{u\}}))$.
If $U\setminus\{u\}=V_0$, then the proof is finished. 
If $U\setminus\{u\}\supsetneq V_0$, we repeat the same argument until we obtain $\mathrm{Hull}(orb(A_U))\subset \mathrm{Hull}(orb(A_{V_0}))$.

When $\#(T\setminus U)>0$, there exists $s\in T\setminus U$.
Then $C\cap sC$ contains a face spanned by $A_{U\sqcup \{s\}}$ and $A_T$.
The geodesic from $x$ to $sx$ passes through this face. Let $y$ be the intersection point.
Then $y\in \mathrm{Hull}(orb(x))$, and thus $\mathrm{Hull}(orb(y))\subset \mathrm{Hull}(orb(x))$.
By induction, $\mathrm{Hull}(orb(y))=C_{\Gamma,\Gamma_0}$.
\end{proof}

The following is a generalization of \cite[Proposition 4.8]{KO2}. 
\begin{proposition}\label{finite normal subgroup}
Suppose that $\Gamma$ is not a clique. 
Then, any finite normal subgroup of $A_\Gamma$ is contained in $A_{V_0}$. 
In particular, any finite subgroup of the center of $A_\Gamma$ is contained in $A_{V_0}$.
Also, if $A_\Gamma$ is decomposed as a direct product $A_{\Gamma}=A_1\times A_2$ and 
$A_1$ is finite, then $A_1$ is contained in $A_{V_0}$ and we have $A_{V_0}=A_1\times (A_{V_0}\cap A_2)$. 
\end{proposition}
\begin{proof}
Let $N$ be a finite normal subgroup of $A_\Gamma$. 
Set $$Fix(N):=\{x\in C_{\Gamma,\Gamma_0} \mid nx=x\text{ for any }n\in N \}.$$ 
Because $N$ is finite and $C_{\Gamma,\Gamma_0}$ is a complete CAT(0) space, we have $Fix(N)\neq \emptyset$.
Take any $x\in Fix(N)$. 
Then, $orb(x)=A_\Gamma x\subset Fix(N)$. 
Indeed, the normality of $N$ implies that, for any $g\in A_\Gamma$ and $n\in N$, there exists $n'\in N$ 
such that $ng=gn'$. Therefore, we have $ngx=gn'x=gx$.
Because $C_{\Gamma,\Gamma_0}$ is CAT(0), $Fix(N)$ is convex. Hence, we have $Hull(orb(x))\subset Fix(N)$. 
By Lemma \ref{CM Lemma 3.2}, we have $Hull(orb(x))=C_{\Gamma,\Gamma_0}$. 
Hence, $Fix(N)=C_{\Gamma,\Gamma_0}$. In particular, $A_{V_0} \in Fix(N)$. 
Hence, $N$ must be contained in $A_{V_0}$.
\end{proof}

Finally, we state a lemma related to the local geometry of reduced clique-cube complexes. 
The following is a nontrivial variant in the setting of $C_{\Gamma,\Gamma_0}$ 
of \cite[Lemma 5.2]{KO2} concerning $C_{\Gamma}$. 
\begin{lemma}\label{twistnew}
Let $s,t\in V(\Gamma_\ast)$ belong to mutually different indecomposable factors of $\Gamma_\ast$. 
Suppose that we have $d\in \mathbb N$ and mutually different $w_0, w_1, \ldots, w_{d}\in V$ 
such that $w_0=s$, $w_d=t$ and $(w_{i-1}, w_i)\in E$ with label greater than $2$ for any $i\in\{1,\ldots,d\}$.
Moreover, suppose that $w_i\in V_0$ if $0<i<d$.
Set $\tau:=w_0w_1\cdots w_d\in A_\Gamma$. 
Let $U\in \Delta_{\Gamma,\Gamma_0}$ with $s,t\in U$. 
Then, there exists no square in $C_{\Gamma,\Gamma_0}$ 
containing both edges $[A_U,A_{U\setminus\{s\}}]$ and $[A_U,\tau A_{U\setminus\{t\}}]$.
That is, 
there exists no $g\in A_{U\setminus\{s\}}$ such that $gA_{U\setminus\{t\}}=\tau A_{U\setminus\{t\}}$.
\end{lemma}
\noindent See Figure~\ref{square_fig}. 
\begin{figure}
\begin{center}
\includegraphics[width=8cm,pagebox=cropbox,clip]{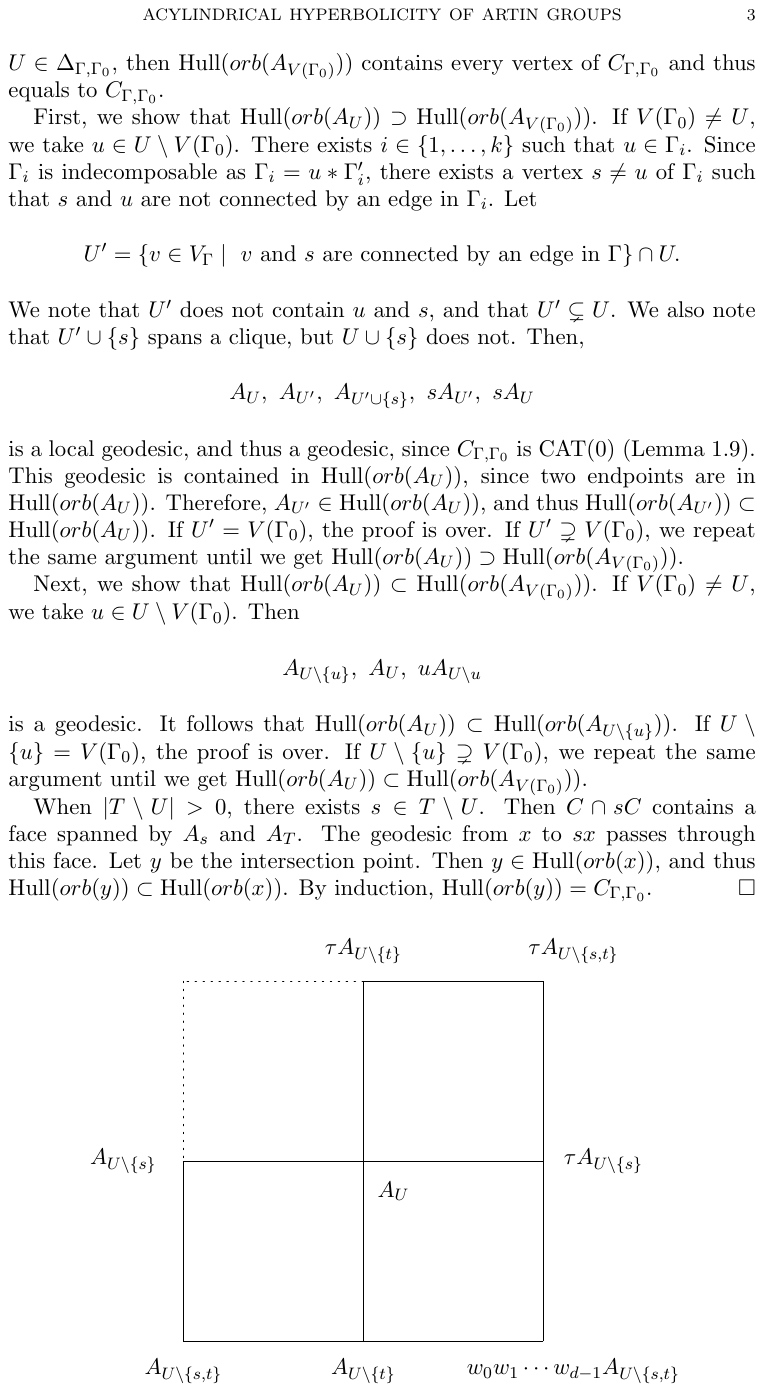} 
\caption{Squares around $A_{U}$ where $s=w_0$, $t=w_d$ and $\tau=w_0w_1\cdots w_d$.}\label{square_fig}
\end{center}
\end{figure}
\begin{proof}
Denote by $\bar{\tau}$ the image of $\tau$ 
by the natural projection $A_\Gamma\to W_\Gamma$. 
Since $w_0,w_1,\ldots,w_d$ are mutually different, 
the word $w_0w_1\cdots w_d$ is a reduced expression of $\bar{\tau}$. 
Moreover this is a unique reduced word representing $\bar{\tau}$ 
because $\bar{\tau}$ is regarded 
as a Coxeter element in $W_{\{w_0,w_1,\ldots,w_d\}}$ and two elements $w_{i-1}$ 
and $w_i$ cannot commute for any $i\in\{1,\ldots,d\}$ 
(see, for example, \cite{Shi}). 

Assume that 
we have $g\in A_{U\setminus\{s\}}$ such that $gA_{U\setminus\{t\}}=\tau A_{U\setminus\{t\}}$. 
Then we have $\tau \in A_{U\setminus\{s\}}A_{U\setminus\{t\}}$, and thus 
$\bar{\tau}\in W_{U\setminus\{s\}}W_{U\setminus\{t\}}$.
Take $a_1,\ldots, a_{d_1}\in U\setminus\{s\}$ and $b_1,\ldots, b_{d_2}\in U\setminus\{t\}$ 
such that the word $a_1\cdots a_{d_1}b_1\cdots b_{d_2}$ represents $\bar{\tau}$. 
Then, by the deletion property for Coxeter groups (see, for example, \cite{BB}), the word 
$a_1\cdots a_{d_1}b_1\cdots b_{d_2}$ contains a reduced expression for $\bar{\tau}$ as a subword. 
In any such subword, the letter $t$ precedes $s$. 
This contradicts the fact that $w_0=s$ precedes $w_d=t$ 
in the unique reduced expression $w_0w_1\cdots w_d$ of $\bar{\tau}$. 
\end{proof}

\section{Proof of Theorem~\ref{main}}\label{main_proof}
We show our main theorem (Theorem~\ref{main}).
In fact, we prove the following:
\begin{theorem}\label{main-4}
Let $A_\Gamma$ be an Artin group associated with $\Gamma$. 
Suppose that $\Gamma$ is not a clique. 
Then, the following are equivalent:
\begin{enumerate}
\item $A_\Gamma$ is irreducible; 
\item $A_\Gamma$ has a WPD contracting element with respect to the isometric action on the reduced clique-cube complex;
\item $A_\Gamma$ is acylindrically hyperbolic; 
\item $A_\Gamma$ is virtually directly indecomposable; 
that is, it cannot be decomposed as a direct product 
of two infinite subgroups.
\end{enumerate}
\end{theorem}
In this section, let $A_\Gamma$ be an Artin group associated with a defining graph $\Gamma$ with the vertex set $V(\Gamma)=V$ and the edge set $E(\Gamma)=E$,
as in Section~\ref{defining_graph}. 
We suppose that $\Gamma$ is not a clique. 

\subsection{Proof of $(2)\Rightarrow(3),(3)\Rightarrow(4),(4)\Rightarrow(1)$ in Theorem~\ref{main-4}}\label{part}
We show that $(2)\Rightarrow(3),(3)\Rightarrow(4),(4)\Rightarrow(1)$ in Theorem~\ref{main-4}. 

First, $(2)\Rightarrow(3)$ follows from Theorem~\ref{BBF}. 

Next, we show $(3)\Rightarrow(4)$. 
Let $A_\Gamma$ be acylindrically hyperbolic. If $A_\Gamma$ is isomorphic to a direct product $A_1\times A_2$, then 
either $A_1$ or $A_2$ is finite by acylindrical hyperbolicity \cite[Corollary 7.3]{Osin}. 

Finally, because every non-trivial Artin group is infinite, $(4)\Rightarrow(1)$ is clear. 

\subsection{Proof of $(1)\Rightarrow(2)$ in Theorem~\ref{main-4}}
We give a proof of $(1)\Rightarrow(2)$ in Theorem~\ref{main-4}. 
Suppose that $A_\Gamma$ is irreducible; that is, $\Gamma^t$ is connected. 
By noting that $\Gamma$ is not a clique, we consider the unique join decomposition $\Gamma=\Gamma_0\ast\Gamma_\ast=\Gamma_0\ast(\ast_{1\le i\le k}\Gamma_i)$ 
where $\Gamma_0$ is the maximal clique factor and each $\Gamma_i$ ($1\le i\le k$) 
is an indecomposable factor with at least two vertices 
(see Definition~\ref{decompositon}). 
We set $V_i=V(\Gamma_i)$ and $E_i=E(\Gamma_i)$ for each $i\in\{0,1,\ldots,k\}$. 
Also, we set $V_{\ast}=V(\Gamma_{\ast})$. 
Then, for every $i\neq 0$, the complement graph $(\Gamma_i)^c$ of $\Gamma_i$ 
is connected and the complement graph $(\Gamma_\ast)^c$ of $\Gamma_\ast$ is 
the disjoint union of connected components: $(\Gamma_\ast)^c=\bigsqcup_{1\le i \le k}(\Gamma_i)^c$, by Lemma~\ref{join-decomposition_lem}.
When $k=1$ and $\Gamma_0=\emptyset$, $C_{\Gamma,\Gamma_0}=C_{\Gamma}$ and a WPD contracting element is given in \cite{Charney}. 
When $k>1$ and $\Gamma_0=\emptyset$, a WPD contracting element is given in \cite{KO2}. 
Our goal in this section is to give a WPD contracting element with respect to the action on the reduced clique-cube complex $C_{\Gamma,\Gamma_0}$ in our general setting. 
Indeed, we explicitly construct three elements: $\gamma_{\natural}\in A_{V_{\ast}}$, and $\gamma_{\sharp}, \gamma_{\flat}\in A_{\Gamma}$ which have some features aligned with Theorem~\ref{criterion}, and show that their product $\gamma_{\natural\sharp\flat}=\gamma_{\natural}\gamma_{\sharp}\gamma_{\flat}\in A_{\Gamma}$ is indeed a WPD contracting element.
It may be useful for understanding the construction of $\gamma_{\natural\sharp\flat}$ to look at Section~\ref{example_sec}, 
which treats a concrete example.

We define $Q(\Gamma,\Gamma_0)$ as a finite simple graph with 
\begin{align}
\begin{aligned}\label{QGamma}
&V(Q(\Gamma,\Gamma_0))=\{V_i\mid 1\leq i\leq k\},\\
&E(Q(\Gamma,\Gamma_0))=\left\{(V_i,V_j) \middle|
\begin{array}{l}
1\leq i,\ j\leq k,\ i\neq j,\\
(\Gamma_i)^t\text{ and }(\Gamma_j)^t\text{ are connected in }(\Gamma_0\ast\Gamma_i\ast\Gamma_j)^t
\end{array}
\right\}.
\end{aligned}
\end{align}
Note that for $i\ne 0$, $(\Gamma_i)^t$ is connected because $(\Gamma_i)^c$ is connected, 
$V((\Gamma_i)^c)=V((\Gamma_i)^t)=V_i$, and $E((\Gamma_i)^c)\subset E((\Gamma_i)^t)$. 
Note that the following are equivalent for different $i, j\in \{1,\ldots,k\}$: 
\begin{itemize}
\item[(I)] $(\Gamma_i)^t$ and $(\Gamma_j)^t$ are connected in $(\Gamma_0\ast\Gamma_i\ast\Gamma_j)^t$; 
\item[(II)] there exists a path 
from a vertex of $\Gamma_i$ and a vertex of $\Gamma_j$ on $\Gamma_0\ast\Gamma_i\ast\Gamma_j$,
consisting of edges with label greater than $2$.
\end{itemize}
Because $A_\Gamma$ is irreducible, i.e., $\Gamma^t$ is connected, 
$Q(\Gamma,\Gamma_0)$ is connected. 

We take a spanning tree $T$ of $Q(\Gamma,\Gamma_0)$. 
We regard $T$ as a rooted tree with the root $V_1$.
By trading indices of $V_2,\ldots,V_k$ if necessary, 
we suppose that $i<j$ only if $V_i$ is not farther than $V_j$ from $V_1$ in $T$. 

For each $i,j$ with $i<j$ and $(V_i,V_j)\in E(T)$, 
take a path on $\Gamma$;
\begin{align}\label{p_ij}
p_{i,j}=(w_{i,j,0},w_{i,j,1},\ldots,w_{i,j,d_{i,j}})
\end{align}
consisting of edges with label greater than $2$
such that 
$w_{i,j,0}\in V_i$, $w_{i,j,d_{i,j}}\in V_j$ and
$w_{i,j,1},\ldots,w_{i,j,d_{i,j}-1}\in V_0$.
Moreover, set $d_{j,i}:=d_{i,j}$ and a path on $\Gamma$;
\begin{align}\label{p_ji}
p_{j,i}=(w_{j,i,0},w_{j,i,1},\ldots,w_{j,i,d_{j,i}}):=(w_{i,j,d_{i,j}},\ldots,w_{i,j,1},w_{i,j,0}).
\end{align}
For any $i,j\in \{1,\ldots,k\}$ with $(V_i,V_j)\in E(T)$, 
set 
\begin{align}\label{st_ij}
s_{i,j}:=w_{i,j,0}\in V_i,\ t_{i,j}:=w_{i,j,d_{i,j}}\in V_j
\end{align} 
and 
\begin{align}\label{tau_ij}
\tau_{i,j}=w_{i,j,0}w_{i,j,1}\cdots w_{i,j,d_{i,j}}\in A_\Gamma.
\end{align} 
Then we have
$s_{j,i}=t_{i,j}$ in $V_j$, $t_{j,i}=s_{i,j}$ in $V_i$ and 
$\tau_{j,i}=w_{i,j,d_{i,j}}\cdots w_{i,j,1}w_{i,j,0}\in A_{\Gamma}.$

When $k\geq 2$, we take a closed path 
\begin{equation}\label{pathtree}
(V_{i_1},\dots,V_{i_r},V_{i_{r+1}})
\end{equation}
with $i_1=i_{r+1}=1$ 
on the spanning tree $T$ of $Q(\Gamma,\Gamma_0)$ 
passing through every vertex at least once. 
Since $r$ is the length of the closed path on the tree, $r$ is an even number.
If $k=1$, then $T$ has no edges, and we set $r=0$.

To construct $\gamma_{\natural}$ and $\gamma_{\flat}$,
we start with the following lemma, which is almost the same as \cite[Lemma 6.1]{KO2}. 
See Section~\ref{matrix} for notations related to matrices.
\begin{lemma}\label{closed}
There exist $n\in \mathbb{N}$, a matrix $M_{\natural}=(v_{i,l})\in \mathcal{M}(k, n; V)$, and a map 
\begin{equation}\label{l-flat_eq}
l_{\flat}:\left\{(i,j)\middle| \begin{array}{l}
 1\leq i,j\leq k,\\
 (V_i,V_j) \in E(T)
 \end{array}
\right\}
\to \{1,\ldots, n\}
\end{equation}
which satisfy the following:
\begin{itemize}
\item[(i)] 
For each $i\in \{1,\ldots,k\}$, $(v_{i,1},\ldots, v_{i,n}, v_{i,1})$ is a closed path on $(\Gamma_{i})^c$ that passes through every vertex at least once;
\item[(ii)]
For each $i, j\in \{1,\ldots,k\}$ with $(V_i,V_j)\in E(T)$, 
$v_{i,l_{\flat}(i,j)}=s_{i,j}$, $v_{j,l_{\flat}(i,j)}=t_{i,j}$;
\item[(iii)] $l_{\flat}({j,i})=l_{\flat}(i,j)$. 
\end{itemize}
\end{lemma}

\begin{proof}
For any $i\in \{1,\ldots,k\}$, 
we consider the length $n_i$ of a closed path
on $(\Gamma_{i})^c$ passing through every vertex at least once. 
Set $n=\prod_{1\le i\le k}n_i$. 
Then, for any $i\in \{1,\ldots,k\}$, 
by concatenating $\frac{n}{n_i}$ copies 
of a closed path of length $n_i$ on $(\Gamma_{i})^c$ 
passing through every vertex at least once, 
we have a closed path $(v'_{i,1},\ldots,v'_{i,n},v'_{i, n+1})$ 
with $v'_{i,1}=v'_{i,n+1}$ on $(\Gamma_{i})^c$ 
passing through every vertex at least $\frac{n}{n_i}$ times (in particular, at least once). 

We set $(v_{1,1},\ldots,v_{1, n},v_{1, n+1}):=(v'_{1,1},\ldots,v'_{1, n},v'_{1, n+1})$.
For any $j\in \{2,\ldots,k\}$, we define 
$(v_{j,1},\ldots,v_{j, n},v_{j, n+1})$ inductively as follows. 
Take $j\in \{2,\ldots,k\}$. 
Suppose that $(v_{i,1},\ldots,v_{i, n},v_{i, n+1})$ 
is defined for $i\in \{1,\ldots, j-1\}$ with $(V_i,V_j)\in E(T)$. 
Note that such $i$ is unique for $j$.
Then, we define $l_{\flat}(i,j)$ as the minimum $l$ such that $v_{i,l}=s_{i,j}$. 
By a cyclic permutation of  $v'_{j,1},\ldots,v'_{j, n}$, 
we have $v_{j,1},\ldots,v_{j, n}$ such that $v_{j,l_{\flat}(i,j)}=t_{i,j}$. 
By setting $v_{j,n+1}:=v_{j,1}$, we have 
$(v_{j,1},\ldots,v_{j, n},v_{j, n+1})$. 
Finally, for any $i,j$ with $1\le i<j\le k$ and $(V_i,V_j)\in E(T)$, we set $l_{\flat}(j,i):=l_{\flat}(i,j)$. 
\end{proof}
We take $n\in \mathbb{N}$, $M_{\natural}=(v_{i,l})\in \mathcal{M}(k,n;V)$ and a map $l_{\flat}$ as in Lemma~\ref{closed}.
In the case where $k=1$, $E(T)=\emptyset$,
and thus $l_{\flat}$ and the conditions (ii), (iii) are meaningless. 
Nonetheless, the condition (i) remains meaningful.

Now we define $\gamma_\natural\in A_{V_{\ast}}$ as 
\begin{equation}\label{gammaelement^0}
\gamma_\natural:=prod(M_{\natural})=(v_{1,1}v_{2,1}\cdots v_{k,1})\cdots (v_{1,n}v_{2,n}\cdots v_{k,n}).
\end{equation}
When $\Gamma_0=\emptyset$ and $k=1$, this $\gamma_{\natural}$ is a WPD contracting element on $C_{\Gamma,\Gamma_0}=C_{\Gamma}$ \cite{Charney}.
We do not directly need this fact.
Even in the general setting, we can confirm that $\gamma_{\natural}$ fits the condition (i) of Theorem~\ref{criterion}.
Unfortunately, it is unclear whether $\gamma_{\natural}$ fits the conditions (ii) and (iii).
To more easily approach these conditions, we prepare two modifications $\gamma_{\sharp}$ and $\gamma_{\flat}$ based on powers of $\gamma_{\natural}$.

We prepare for the definition of $\gamma_\flat$.
Let $e$ be the identity in $A_{\Gamma}$. 
When $k=1$ (for example, $\Gamma$ in Figure~\ref{new_ex_fig}), 
we define $\gamma_{\flat}=e$.
Suppose that $k\geq 2$. 
Recall that we have $n\in \mathbb{N}$, $M_{\natural}=(v_{i,l})\in \mathcal{M}(k, n; V)$,
a map $l_{\flat}$ as in Lemma~\ref{closed},
and a closed path $(V_{i_1},\dots,V_{i_r},V_{i_{r+1}})$ with $i_1=i_{r+1}=1$, defined in (\ref{pathtree}).
Note that $r\geq 2$, since $k\geq 2$.
We define supplementary matrices $M'_{\flat}(a), M''_{\flat}(a)\in \mathcal{M}(k, n; A_{\Gamma})$ $(a\in\{1,\ldots, r\})$ by
\begin{equation}\label{M_natural_a_1}
\left(M'_{\flat}(a)\right)[i,l]=
\begin{cases}
v_{i,l}^{-1} \quad &(i\in\{i_a,i_{a+1}\}, l=l_{\flat}(i_a,i_{a+1}))\\
e\quad &(\text{otherwise}),
\end{cases}
\end{equation}
\begin{equation}\label{M_natural_a_2}
\left(M''_{\flat}(a)\right)[i,l]=
\begin{cases}
\tau_{i_a,i_{a+1}} \quad &(i=1, l=l_{\flat}(i_a,i_{a+1}))\\
e\quad &(\text{otherwise}).
\end{cases}
\end{equation}
Using $M'_{\flat}(a)$ and $M''_{\flat}(a)$, we define $M_{\flat}(a)\in \mathcal{M}(k, n; A_{\Gamma})$ by
\begin{equation}\label{M_natural_a}
M_{\flat}(a):=M''_{\flat}(a)\odot M'_{\flat}(a)\odot M_{\natural},
\end{equation}
where $\odot$ is the Hadamard product of matrices (see $(\ref{odot_map})$).
$M_{\flat}(a)$ can be considered as a modification of $M_{\natural}$.
We define $M_{\flat}\in \mathcal{M}(k, nr; A_{\Gamma})$ by
\begin{equation}\label{M_natural}
M_{\flat}=cbind\left(M_{\flat}(1),\ldots ,M_{\flat}(r)\right).
\end{equation}
For the definition of $cbind$, see (\ref{cbind-def}).

We now define $\gamma_\flat\in A_{\Gamma}$ as 
\begin{equation}\label{gammaelement'}
\gamma_\flat:=prod(M_{\flat})=prod(M_{\flat}(1))\cdots prod(M_{\flat}(r)).
\end{equation}
For the definition of the map $prod$, see (\ref{prod1_map}) and (\ref{prod2_map}).
$\gamma_{\flat}$ is a modification of $\gamma_{\natural}^r$.
When $k\geq 2$ and $\Gamma_0= \emptyset$, we can take $\gamma_{\flat}$ as a WPD element \cite{KO2}.
This fact is not used directly.
To deal with the case where $\Gamma_0\neq \emptyset$, we need the other modification, $\gamma_{\sharp}$.

To explain the construction of $\gamma_{\sharp}$, we introduce the ``height'' $h$ of elements of $V_0$ in the case $V_0\neq \emptyset$.
Otherwise, we define $\gamma_{\sharp}=e$.
For $h\ge 1$, we set 
\begin{equation}\label{W(h)}
W(h):=\{w\in V_0 \mid \min_{v\in V_\ast}d_{\Gamma^t}(w,v)=h\}.
\end{equation} 
Since $\Gamma^t$ is connected, we have $\mathfrak{h}\in \mathbb N$ such that  $W(\mathfrak{h})\neq \emptyset$ 
and $V_0=W(1)\sqcup \cdots \sqcup W(\mathfrak{h})$. 
Suppose that $\mathfrak{h}\geq 2$. Otherwise, we define $\gamma_{\sharp}=e$.
We define two maps 
\begin{align}
&i_{\sharp}: V_0\to \{1,\ldots, k\},\label{i_flat} \\ 
&l_{\sharp}: V_0\to \{1,\ldots, n\}\label{l_flat}
\end{align}
such that $(i_{\sharp}(w),l_{\sharp}(w))$ is the minimum of $(i,l)$ indexing $v_{i,l}\in V_*$ with 
$$d_{\Gamma^t}(w,v_{i,l})=\min_{v\in V_\ast}d_{\Gamma^t}(w,v)$$
with respect to the order $<$ on $\{1,\ldots,k\}\times\{1,\ldots,n\}$
defined by
$$(i,l)< (i',l')\Leftrightarrow (l<l')\text{ or } (l=l'\text{ and }  i< i').$$ 
Note that this is a slightly unusual dictionary order, which considers the second factor first.
In particular, for $w\in W(1)$, $(i_{\sharp}(w),l_{\sharp}(w))$ is the minimum of $(i,l)$ with $\mu(w,v_{i,l})>2$ with respect to the above order.

We set 
\begin{equation}\label{W(0)}
W(0):=\{v_{i_{\sharp}(w),l_{\sharp}(w)}\in V_\ast \mid w\in W(1)\},
\end{equation}
which equals $\{v_{i_{\sharp}(w),l_{\sharp}(w)}\in V_\ast \mid w\in V_0\}$. 
Obviously, we have $W(0)\cap V_0=\emptyset$.
We extend two maps $i_{\sharp}$ and $l_{\sharp}$ on $W(0)\sqcup V_0$ by defining
$(i_{\sharp}(w),l_{\sharp}(w))$ for $w\in W(0)$ as the minimum of $(i,l)$ with $w=v_{i,l}$ 
with respect to the order $<$ on $\{1,\ldots,k\}\times\{1,\ldots,n\}$.
Note that for $w\in W(0)$, $i$ and $l$ with $w=v_{i,l}$, $i$ is unique, but $l$ is not necessarily unique. 

Next, by using the order $<$ on $\{1,\ldots,k\}\times\{1,\ldots,n\}$ 
and the two maps $i_{\sharp}$, $l_{\sharp}$ on $W(0)\sqcup V_0=W(0)\sqcup W(1)\sqcup \cdots \sqcup W(\mathfrak{h})$,
we define a total order $\prec$ on $W(0)\sqcup V_0$ as follows. 
First, we take a total order $\prec$ on $W(0)$ such that 
for any $w,w'\in W(0)$, $w\prec w'$ if and only if  $(i_{\sharp}(w),l_{\sharp}(w))< (i_{\sharp}(w'),l_{\sharp}(w'))$.
Next, suppose that for $1\leq h\leq \mathfrak{h}$, $W(0)\sqcup\cdots\sqcup W(h-1)$ is equipped with a total order $\prec$. 
Then we define a map 
$$\pi_h:W(h)\to W(h-1)$$
as $\pi_h(w)$ $(w\in W(h))$ is the minimum of $w'\in W(h-1)$ with $\mu(w,w')>2$ 
with respect to the given total order $\prec$. 
We take a total order $\prec$ on $W(h)$ such that 
for any different $w,w'\in W(h)$, $w\prec w'$ only if $\pi(w)\preceq \pi(w')$.
Moreover, for any $w\in W(0)\sqcup\cdots\sqcup W(h-1)$ and $w'\in W(h)$, 
we set $w\prec w'$. 
Then $\prec$ on $W(0)\sqcup\cdots\sqcup W(h-1)$ is extended on $W(0)\sqcup\cdots\sqcup W(h)$. 
Finally, we have the order $\prec$ on $W(0)\sqcup V_0=W(0)\sqcup \cdots \sqcup W(\mathfrak{h})$. 
We define a map
\begin{equation}\label{w(h)}
\pi:W(1)\sqcup\cdots\sqcup W(\mathfrak{h})\to W(0)\sqcup\cdots\sqcup W(\mathfrak{h}-1),
\end{equation}
by $\pi(w)=\pi_{h}(w)\in W(h-1)$ for $w\in W(h)$.
For any $w\in W(0)\sqcup V_0$, let $h\in \{0,\ldots,\mathfrak{h}\}$ be the height of $w$; that is, $w\in W(h)$. 
We take a path from $w$ to $v_{i_{\sharp}(w), l_{\sharp}(w)}$ on $\Gamma^t$;
\begin{align}\label{w(h)_path}
(w(h),w(h-1),\ldots,w(0)):=(w, \pi(w),\pi^2(w),\ldots, \pi^{h}(w)).
\end{align}
We note that $(i_{\sharp}(w),l_{\sharp}(w))=(i_{\sharp}(w(h')),l_{\sharp}(w(h')))$ for any $0\le h'\le h$. 
In particular, $w(0)=\pi^h(w)=v_{i_{\sharp}(w),l_{\sharp}(w)}$. \noindent See Figure~\ref{W_fig}. 
\begin{figure}
\begin{center}
\includegraphics[width=13cm,pagebox=cropbox,clip]{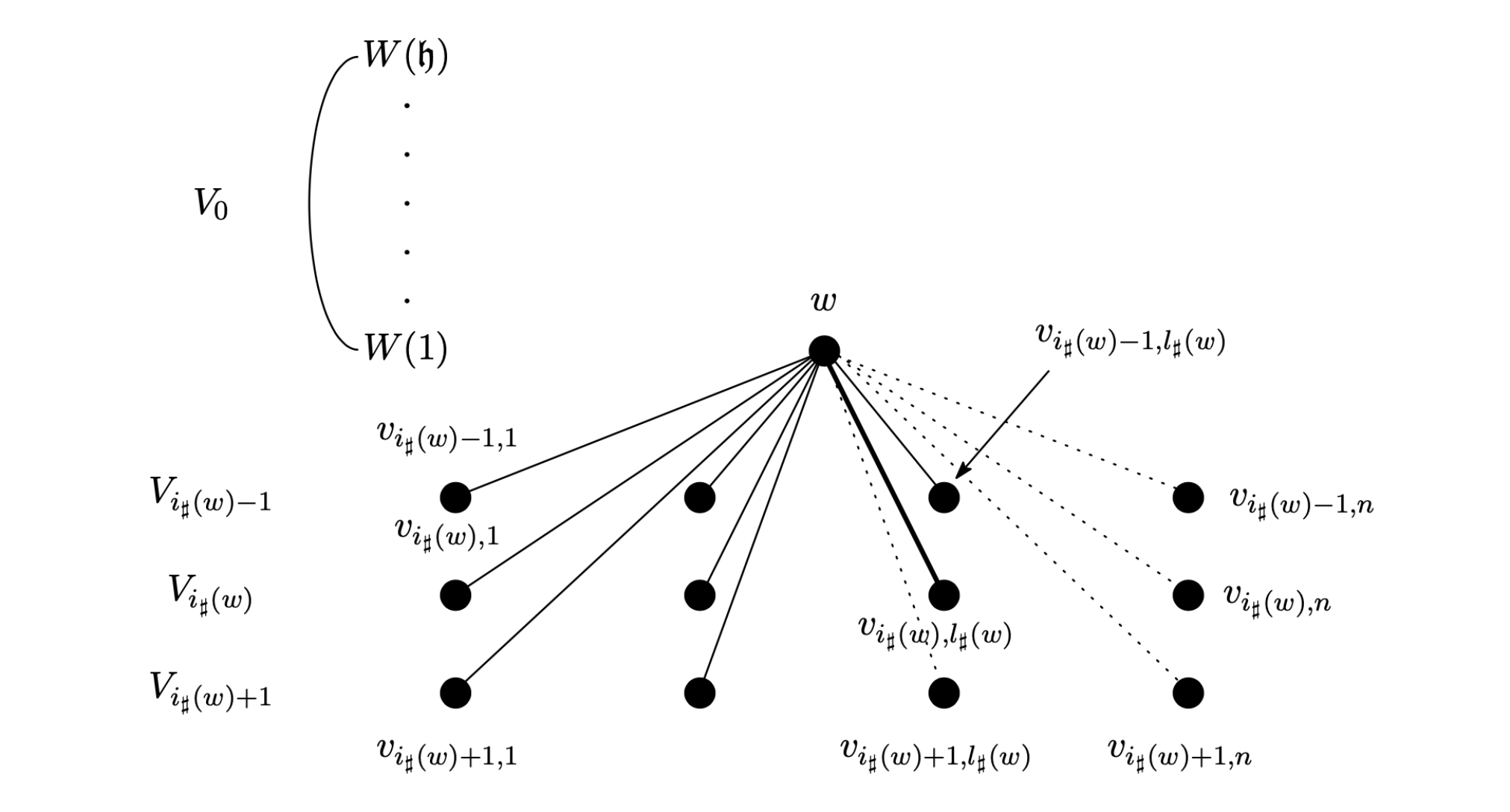}
\caption{In the figure, $w$ is a vertex in $W(1)$, and thin solid lines indicate edges labeled by $2$.
A thick solid line, which connects $w$ and $v_{i_{\sharp}(w), l_{\sharp}(w)}$, indicates an edge labeled by an integer greater than $2$. 
For all edges which are presented as dotted lines, we do not care about their label.
}\label{W_fig}
\end{center}
\end{figure}
Now we name all elements of $W(1)\sqcup \cdots \sqcup W(\mathfrak{h}-1)=V_0\setminus W(\mathfrak{h})$ 
as 
\begin{equation}\label{index_total-order}
w_1,w_2,\ldots,w_{\mathfrak{m}}
\end{equation} 
such that 
$w_1\prec w_2\prec \cdots \prec w_{\mathfrak{m}}$ where 
\begin{equation}\label{density_total-order}
\mathfrak{m}:=\#(W(1)\sqcup \cdots \sqcup W(\mathfrak{h}-1)). 
\end{equation}
Moreover, we set the heights $h_1, h_2,\ldots, h_{\mathfrak{m}}$ for $w_1, w_2, \ldots, w_{\mathfrak{m}}$, respectively.
We have $1=h_1\leq h_2\leq \cdots \leq h_{\mathfrak{m}}=\mathfrak{h}-1.$
The following is clear by the construction of the order $\prec$ and the map $\pi$.
\begin{lemma}\label{sharp_lem}
Let $h\in \{1,\ldots, \mathfrak{h}\}$ and $w\in W(h)$.
\begin{enumerate}
\item Let $h=1$. 
Let $w'\in V_0$ with $w'\succeq w$. 
For any $(i,l)\in \{1,\ldots,k\}\times\{1,\ldots,n\}$ with $(i,l)<(i_{\sharp}(w), l_{\sharp}(w))$, we have $\mu(v_{i,l},w')=2$.
\item Let $h\geq 2$. Let $w'\in V_0$ with $w'\succeq w$. \begin{enumerate}
\item[(i)] For any $v\in V_{\ast}$, we have $\mu(v,w')=2$. 
\item[(ii)] For any $w''\in V_0$ with $w''\prec \pi(w)$, we have $\mu(w'',w')=2$.
\end{enumerate}
\end{enumerate}
\end{lemma}
\noindent See Figure~\ref{W_fig}.

We now explain the construction of $\gamma_{\sharp}$.
If $V_0=\emptyset$ or $\mathfrak{h}=1$,
then we have already defined $\gamma_{\sharp}=e$.
Suppose that $V_0\neq \emptyset$ and that $\mathfrak{h}\geq 2$.
We define supplementary matrices
$M'_{\sharp}(\mathfrak{j}), M''_{\sharp}(\mathfrak{j}) \in \mathcal{M}(k,n;A_{\Gamma})$ $(\mathfrak{j}\in \{1,\ldots,\mathfrak{m}\})$ by
\begin{equation}\label{M_flat_j_1}
\left(M'_{\sharp}(\mathfrak{j})\right)[i,l]=
\begin{cases}
v_{i,l}^{-1} &\quad(i=i_{\sharp}(w_{\mathfrak{j}}), l=l_{\sharp}(w_{\mathfrak{j}}))\\
e &\quad(\text{otherwise})
\end{cases}
\end{equation}
and
\begin{equation}\label{M_flat_j_2}
\left(M''_{\sharp}(\mathfrak{j})\right)[i,l]=
\begin{cases}
w_{\mathfrak{j}}\left(h_{\mathfrak{j}}\right) w_{\mathfrak{j}}\left(h_{\mathfrak{j}}-1\right)\cdots w_{\mathfrak{j}}(0) &\quad(i=k, l=l_{\sharp}(w_{\mathfrak{j}}))\\
e &\quad(\text{otherwise}),
\end{cases}
\end{equation}
where $h_{\mathfrak{j}}$ is the height of $w_{\mathfrak{j}}$.
Using $M'_{\sharp}(\mathfrak{j})$ and $M''_{\sharp}(\mathfrak{j})$, we define $M_{\sharp}(\mathfrak{j})\in \mathcal{M}(k, n; A_{\Gamma})$ by
\begin{equation}\label{M_flat_j}
M_{\sharp}(\mathfrak{j}):=M_{\natural}\odot M'_{\sharp}(\mathfrak{j})\odot M''_{\sharp}(\mathfrak{j}),
\end{equation}
where $\odot$ is the Hadamard product of matrices (see $(\ref{odot_map})$).
$M_{\sharp}(\mathfrak{j})$ can be considered as a modification of $M_{\natural}$.
We define $M_{\sharp}\in \mathcal{M}(k, n\mathfrak{m}; A_{\Gamma})$ by 
\begin{equation}\label{M_flat}
M_{\sharp}=cbind\left(M_{\sharp}(1),\ldots ,M_{\sharp}(\mathfrak{m})\right).
\end{equation}
For the definition of $cbind$, see (\ref{cbind-def}).
We now define $\gamma_\sharp\in A_{\Gamma}$ as 
\begin{equation}\label{gammaelement''}
\gamma_\sharp:=prod(M_{\sharp})=prod(M_{\sharp}(1))\cdots prod(M_{\sharp}(\mathfrak{m})).
\end{equation}
$\gamma_{\sharp}$ is a modification of $\gamma_{\natural}^{\mathfrak{m}}$.

We define $M_{\natural\sharp\flat}\in \mathcal{M}(k, (1+\mathfrak{m}+r)n; A_{\Gamma})$ by 
\begin{align}\label{M_naturalsharpflat}
M_{\natural\sharp\flat}:=cbind(M_{\natural},M_{\sharp},M_{\flat}).
\end{align}
Now, we define $\gamma_{\natural \sharp \flat}\in A_{\Gamma}$ by
\begin{equation}\label{gammaelement}
\gamma_{\natural\sharp\flat}:=prod(M_{\natural\sharp \flat})
=prod(M_{\natural})prod(M_{\sharp})prod(M_{\flat})=\gamma_{\natural}\gamma_{\sharp}\gamma_{\flat}.
\end{equation}

For later reference, we define supplementary elements 
$\lambda_l$ $(l\in\{1,\ldots, n\})$, $\lambda_{\sharp}(\mathfrak{j})$, $\gamma_{\sharp}(\mathfrak{j})$ $(\mathfrak{j}\in\{1,\ldots, \mathfrak{m}\})$, $\lambda_{\flat}(a)$, $\gamma_{\flat}(a)$ $(a\in\{1,\ldots, r\})$, 
$\gamma_{\natural\sharp\flat}[d]\in A_{\Gamma}$ $(d\in \{1,\ldots,(1+r+\mathfrak{m})n\})$.
For every $l\in \{1,\ldots,n\}$, 
we define $\lambda_l\in A_{V_{\ast}}\subset A_\Gamma$ by
\begin{equation}\label{lambda_l}
\lambda_l:=prod(M_{\natural}[\ ,l])=
v_{1,l}v_{2,l}\cdots v_{i,l}\cdots v_{k,l}.
\end{equation} 
For every $\mathfrak{j}\in\{1,\ldots, \mathfrak{m}\}$,
we define $\lambda_{\sharp}(\mathfrak{j}), \gamma_{\sharp}(\mathfrak{j})\in A_{\Gamma}$ by
\begin{align}\label{lambda_flat_j}
\lambda_{\sharp}(\mathfrak{j})&:=prod(M_{\sharp}(\mathfrak{j})[\ ,l_{\sharp}(w_{\mathfrak{j}})]),\\
\gamma_{\sharp}(\mathfrak{j})&:=prod(M_{\sharp}(\mathfrak{j})).
\end{align}
For every $a\in\{1,\ldots, r\}$,
we define $\lambda_{\flat}(a), \gamma_{\flat}(a)\in A_{\Gamma}$ by
\begin{align}\label{lambda_natural_a}
\lambda_{\flat}(a)&:=prod(M_{\flat}(a)[\ ,l_{\flat}(i_a,i_{a+1})]),\\
\gamma_{\flat}(a)&:=prod(M_{\flat}(a)).
\end{align}
Note that $\lambda_{\sharp}(\mathfrak{j})$ and $\lambda_{\flat}(a)$ 
can be regarded as modifications of $\lambda_{l_{\sharp}(w_{\mathfrak{j}})}$ and $\lambda_{l_{\flat}(i_a,i_{a+1})}$, respectively.
Also note that $\gamma_{\sharp}(\mathfrak{j})$ and $\gamma_{\flat}(a)$
can be regarded as modifications of $\gamma_{\natural}$.
For every $d\in \{1,\ldots,(1+\mathfrak{m}+r)n\}$, let
$\gamma_{\natural\sharp\flat}[d]$ be the product of the components of the first $d$ columns of $M_{\natural \sharp \flat}\in \mathcal{M}(k, (1+\mathfrak{m}+r)n; A_{\Gamma})$. That is,
\begin{equation}\label{lambda_flat}
\gamma_{\natural\sharp\flat}[d]:=prod(M_{\natural \sharp \flat}[\ ,1:d]).
\end{equation}

We recall the following for convenience: 
\begin{enumerate}
\item $k$ is the number of indecomposable factors of the unique join decomposition of $\Gamma_\ast$; 
\item $r$ is the length of the closed path (\ref{pathtree}) on $T$;
\item $n$ is the common length of the closed paths on $(\Gamma_i)^c$ for all $i\in\{1,\ldots,k\}$ taken in Lemma~\ref{closed}; 
\item $\mathfrak{m}$ is $\#(W(1)\sqcup \cdots \sqcup W(\mathfrak{h}-1))$ as defined in $(\ref{density_total-order})$. 
\end{enumerate}

For $1\leq l\le n$, set 
$$U_l:=V_0\sqcup \{v_{1,l},v_{2,l},\ldots, v_{k,l}\}.$$
Then, $U_l$ spans a clique containing $\Gamma_0$ in $\Gamma$; that is, $U_l\in \Delta_{\Gamma,\Gamma_0}$. 
Hence, we have a $k$-dimensional cube $[A_{V_0}, A_{U_l}]$ in $C_{\Gamma,\Gamma_0}$. 
The hyperplanes dual to edges of $[A_{V_0},A_{U_l}]$ are of $v_{i,l}$-type ($i\in \{1,\ldots,k\}$).  
We denote such hyperplanes as $H_{i,l}$ ($i\in \{1,\ldots,k\}$). 

The following is a variant for $C_{\Gamma,\Gamma_0}$ of 
\cite[Lemma 6.2]{KO2} for $C_{\Gamma}$. 
\begin{lemma}\label{key0}
For $i\in \{1,\ldots,k\}$, $l\in \{1,\ldots,n\}$, $\mathfrak{j}\in\{1,\ldots,\mathfrak{m}\}$ and $a\in\{1,\ldots,r\}$, we have the following: 
\begin{enumerate}
\item $H_{i,l}\cap H_{i,l+1}=\emptyset$, where we set $H_{i,n+1}:=H_{i,1}$; 
\item $H_{i,l}\cap \lambda_l H_{i,l}=\emptyset$; 
\item $H_{i,l_{\flat}(i_{a},i_{a+1})}\cap \lambda_{\flat}(a)H_{i,l_{\flat}(i_{a},i_{a+1})}=\emptyset$; 
\item $H_{i,l_{\sharp}(w_{\mathfrak{j}})}\cap \lambda_{\sharp}(\mathfrak{j})H_{i,l_{\sharp}(w_{\mathfrak{j}})}=\emptyset$; 
\item $[A_{V_0},A_{U_l}]\cap [A_{V_0},A_{U_{l+1}}]=\{A_{V_0}\}$, where we set $U_{n+1}:=U_{1}$;
\item $[A_{V_0},A_{U_l}]\cap \lambda_l [A_{V_0},A_{U_l}]=\{A_{U_l}\}$;
\item $[A_{V_0},A_{U_{l_{\flat}(i_{a},i_{a+1})}}]\cap \lambda_{\flat}(a)[A_{V_0},A_{U_{l_{\flat}(i_{a},i_{a+1})}}]=\{A_{U_{l_{\flat}(i_{a},i_{a+1})}}\}$.
\item $[A_{V_0},A_{U_{l_{\sharp}(w_{\mathfrak{j}})}}]\cap \lambda_{\sharp}(\mathfrak{j})[A_{V_0},A_{U_{l_{\sharp}(w_{\mathfrak{j}})}}]=\{A_{U_{l_{\sharp}(w_{\mathfrak{j}})}}\}$.
\end{enumerate}
\end{lemma}
\begin{proof}
(1) $H_{i,l}$ and $H_{i,l+1}$ are of $v_{i,l}$-type and $v_{i,l+1}$-type, respectively. 
Note that $v_{i,l}\neq v_{i,l+1}$ implies $H_{i,l}\neq H_{i,l+1}$. 
Because  $(v_{i,l}, v_{i,l+1})\notin E$, 
we have 
$H_{i,l}\cap H_{i,l+1}=\emptyset$
(see Remark~\ref{label}).

(2) Because $[A_{U_l\setminus\{v_{i,l}\}},A_{U_l}]\subset N(H_{i,l})$, 
we have $[\lambda_l A_{U_l\setminus\{v_{i,l}\}},\lambda_l A_{U_l}]\subset \lambda_l N(H_{i,l})$. 
Note that $\lambda_l A_{U_l}=A_{U_l}$ and $\lambda_l A_{U_l\setminus\{v_{i,l}\}}=v_{1,l}\cdots v_{i,l} A_{U_l\setminus\{v_{i,l}\}}$. 
Assume that $H_{i,l}\cap \lambda_l H_{i,l}\neq \emptyset$. 
Because $H_{i,l}$ and $\lambda_l H_{i,l}$ are of $v_{i,l}$-type, we see that $H_{i,l}= \lambda_l H_{i,l}$ (see Remark~\ref{label}). 
Then, we have 
$$[\lambda_l A_{U_l\setminus\{v_{i,l}\}},\lambda_l A_{U_l}]=[v_{1,l}\cdots v_{i,l} A_{U_l\setminus\{v_{i,l}\}},A_{U_l}] \subset N(H_{i,l}).$$
Hence, we obtain $A_{U_l\setminus\{v_{i,l}\}}=v_{1,l}\cdots v_{i,l} A_{U_l\setminus\{v_{i,l}\}}$. 
Hence, $A_{U_l\setminus\{v_{i,l}\}}=v_{i,l}A_{U_l\setminus\{v_{i,l}\}}$.
This means that $v_{i,l}\in A_{U_l\setminus\{v_{i,l}\}}$. 
However, from Lemma~\ref{van} (2), $\{v_{i,l}\}\cap (U_l\setminus\{v_{i,l}\})=\emptyset$ implies that 
$A_{\{v_{i,l}\}}\cap A_{U_l\setminus\{v_{i,l}\}}=A_\emptyset=\{e\}$. 
This contradicts $v_{i,l}\neq e$ in $A_\Gamma$.

(3) When $i\neq i_a, i_{a+1}$, by the same argument as in (2), we see that 
$$H_{i,l_{\flat}(i_{a},i_{a+1})}\cap \lambda_{\flat}(a)H_{i,l_{\flat}(i_{a},i_{a+1})}=\emptyset.$$
Now, assume that $i=i_a$ or $i=i_{a+1}$ and
$$H_{i,l_{\flat}(i_a,i_{a+1})}\cap \lambda_{\flat}(a)H_{i,l_{\flat}(i_{a},i_{a+1})}\neq \emptyset.$$
Note that 
$\lambda_{\flat}(a)A_{U_{l_{\flat}(i_{a},i_{a+1})}}=A_{U_{l_{\flat}(i_{a},i_{a+1})}}$ 
and
$$\lambda_{\flat}(a)A_{U_{l_{\flat}(i_{a},i_{a+1})}\setminus\{v_{i,l_{\flat}(i_{a},i_{a+1})}\}}=\tau_{i_a,i_{a+1}} A_{U_{l_{\flat}(i_{a},i_{a+1})}\setminus\{v_{i,l_{\flat}(i_{a},i_{a+1})}\}}.$$
Because both $H_{i, l_{\flat}(i_{a},i_{a+1})}$ and $\lambda_{\flat}(a) H_{i,l_{\flat}(i_{a},i_{a+1})}$ 
are of $v_{i,l_{\flat}(i_{a},i_{a+1})}$-type, we see that $H_{i,l_{\flat}(i_{a},i_{a+1})}= \lambda_{\flat}(a)  H_{i,l_{\flat}(i_{a},i_{a+1})}$. 
Then, we have 
\begin{equation*}
\begin{split}
[\lambda_{\flat}(a) A_{U_{l_{\flat}(i_{a},i_{a+1})}\setminus\{v_{i,l_{\flat}(i_{a},i_{a+1})}\}},\lambda_{\flat}(a)  A_{U_{l_{\flat}(i_{a},i_{a+1})}}]\\
=[\tau_{i_a,i_{a+1}} A_{U_{l_{\flat}(i_{a},i_{a+1})}\setminus\{v_{i,{l_{\flat}(i_{a},i_{a+1})}}\}},A_{U_{l_{\flat}(i_{a},i_{a+1})}}] \subset N(H_{i,{l_{\flat}(i_{a},i_{a+1})}}).
\end{split}
\end{equation*}
Hence, we have 
$$\tau_{i_a,i_{a+1}}A_{U_{l_{\flat}(i_a,i_{a+1})}\setminus\{v_{i,l_{\flat}(i_a,i_{a+1})}\}}=A_{U_{l_{\flat}(i_a,i_{a+1})}\setminus\{v_{i,l_{\flat}(i_a,i_{a+1})}\}}.$$
This contradicts Lemma~\ref{twistnew}.

(4) is shown by a similar manner to (2). 

Parts (5), (6), (7) and (8) follow from (1), (2), (3) and (4), respectively.
\end{proof}

The following is a variant for $C_{\Gamma,\Gamma_0}$ of 
\cite[Lemma 6.3]{KO2} for $C_{\Gamma}$, and its proof is the same. 
In fact, Lemma~\ref{twistnew} plays the same role that \cite[Lemma 5.2]{KO2} did in the proof of \cite[Lemma 6.3]{KO2}.
\begin{lemma}\label{key}
For $a\in\{1,\ldots,r\}$, we have 
$$H_{i_a,l_{\flat}(i_a,i_{a+1})}\cap \lambda_{\flat}(a)H_{i_{a+1},l_{\flat}(i_a,i_{a+1})}=\emptyset.$$
\end{lemma}

\begin{figure}
\begin{center}
\includegraphics[width=12cm,pagebox=cropbox,clip]{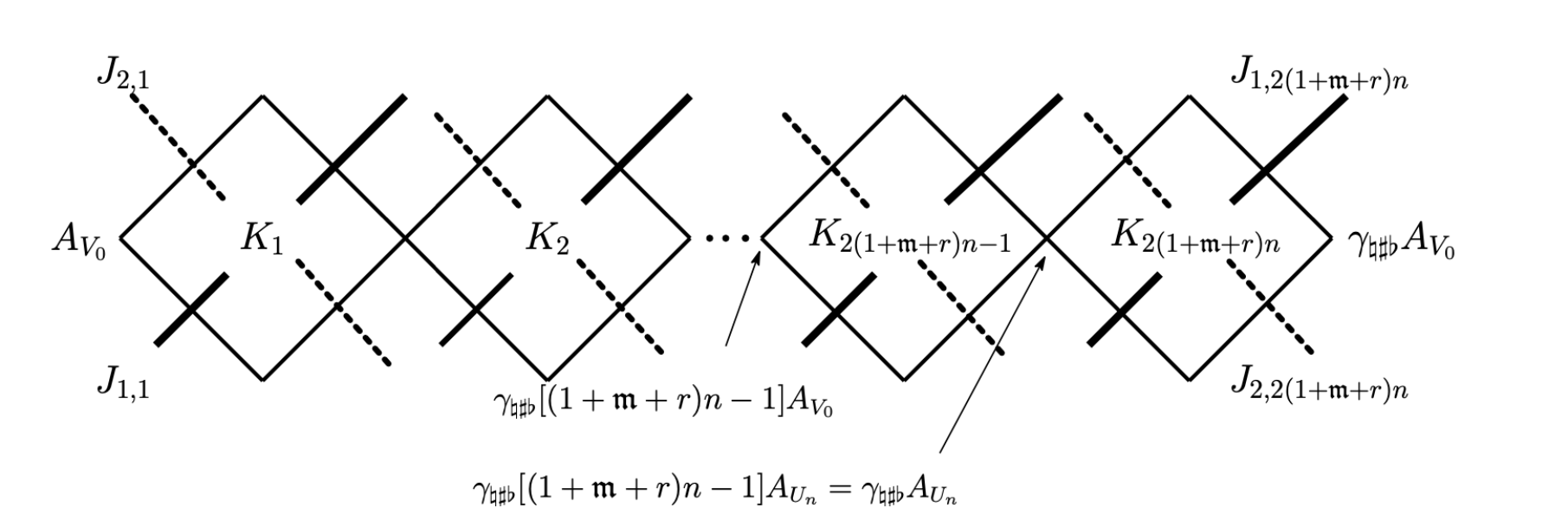}
\caption{All hyperplanes in Lemma~\ref{key3}
for the case where $\Gamma=\Gamma_0\ast\Gamma_1\ast\Gamma_2$. Hyperplanes $J_{1,d}$ and $J_{2,d}$ intersecting $K_d$ are represented as the solid and the dotted lines, respectively.}\label{sequence2_fig}
\end{center}
\end{figure}

Noting Lemma~\ref{key0}, for any $i\in\{1,\ldots,k\}$, we define a sequence of hyperplanes 
$$\ldots,J_{i,-1}, J_{i,0}, J_{i,1},\ldots,J_{i,2(1+\mathfrak{m}+r)n}, J_{i,2(1+\mathfrak{m}+r)n+1}, \ldots,$$ 
and a sequence of $k$-dimensional cubes 
$$\ldots,K_{-1},K_0,K_1,\cdots,K_{2(1+\mathfrak{m}+r)n},K_{2(1+\mathfrak{m}+r)n+1},\ldots$$
as follows.
First, for $a\in\{1,\ldots,1+\mathfrak{m}+r\}$ and $l\in\{1,\dots,n\}$, we define 
\begin{align*}
J_{i,2((a-1)n+l)-1}&:=\gamma_{\natural\sharp\flat}[(a-1)n+l-1]H_{i,l},\\
J_{i,2((a-1)n+l)}&:=\gamma_{\natural\sharp\flat}[(a-1)n+l]H_{i,l},\\
K_{2((a-1)n+l)-1}&:=\gamma_{\natural\sharp\flat}[(a-1)n+l-1][A_{V_0},A_{U_l}]\\
&=[\gamma_{\natural\sharp\flat}[(a-1)n+l-1]A_{V_0},\gamma_{\natural\sharp\flat}[(a-1)n+l-1]A_{U_l}],\\
K_{2((a-1)n+l)}&:=\gamma_{\natural\sharp\flat}[(a-1)n+l][A_{V_0},A_{U_l}]\\
&=[\gamma_{\natural\sharp\flat}[(a-1)n+l]A_{V_0}, \gamma_{\natural\sharp\flat}[(a-1)n+l]A_{U_l}]\\
&=[\gamma_{\natural\sharp\flat}[(a-1)n+l]A_{V_0}, \gamma_{\natural\sharp\flat}[(a-1)n+l-1]A_{U_l}].
\end{align*}
Here, we note that both $J_{i,2((a-1)n+l)-1}$ and $J_{i,2((a-1)n+l)}$ are of the same $v_{i,l}$-type, and that we have 
$\gamma_{\natural\sharp\flat}[(a-1)n+l-1]A_{U_l}=\gamma_{\natural\sharp\flat}[(a-1)n+l]A_{U_l}$.
See Figure~\ref{sequence2_fig}.

Second, for any $b\in\{1,\dots,2(1+\mathfrak{m}+r)n\}$ and $c\in \mathbb{Z}$, we set 
\begin{align*}
J_{i,2(1+\mathfrak{m}+r)nc+b}:={(\gamma_{\natural\sharp\flat})}^cJ_{i,b},\\
K_{2(1+\mathfrak{m}+r)nc+b}:={(\gamma_{\natural\sharp\flat})}^{c}K_{b}.
\end{align*}
Additionally, we have the two connected components 
$J_{i,2(1+\mathfrak{m}+r)nc+b}^-$ and $J_{i,2(1+\mathfrak{m}+r)nc+b}^+$ such that 
$$C_\Gamma \!\setminus\!\!\setminus J_{i,2(1+\mathfrak{m}+r)nc+b}=J_{i,2(1+\mathfrak{m}+r)nc+b}^- \sqcup J_{i,2(1+\mathfrak{m}+r)nc+b}^+,$$
and the ``left endpoint'' of $K_{2(1+\mathfrak{m}+r)nc+b}$ is in $J_{i,2(1+\mathfrak{m}+r)nc+b}^-$ while the ``right endpoint'' of $K_{2(1+\mathfrak{m}+r)nc+b}$ is in $J_{i,2(1+\mathfrak{m}+r)nc+b}^+$. 
Precisely, we have
\begin{align*}
&{(\gamma_{\natural\sharp\flat})}^c \left(\gamma_{\natural\sharp\flat}\left[\frac{b-1}{2}\right]\right)A_{V_0}\in J_{i,2(1+\mathfrak{m}+r)nc+b}^-,\\
&{(\gamma_{\natural\sharp\flat})}^c \left(\gamma_{\natural\sharp\flat}\left[\frac{b-1}{2}\right]\right)A_{U_l}\in J_{i,2(1+\mathfrak{m}+r)nc+b}^+
\end{align*}
if $b=2((a-1)n+l)-1$ with some $a\in\{1,\ldots,1+\mathfrak{m}+r\}$ and $l\in\{1,\dots,n\}$, and we have
\begin{align*}
&{(\gamma_{\natural\sharp\flat})}^c \left(\gamma_{\natural\sharp\flat}\left[\frac{b}{2}\right]\right)A_{U_l}\in J_{i,2(1+\mathfrak{m}+r)nc+b}^-, \\
&{(\gamma_{\natural\sharp\flat})}^c \left(\gamma_{\natural\sharp\flat}\left[\frac{b}{2}\right]\right)A_{V_0}\in J_{i,2(1+\mathfrak{m}+r)nc+b}^+
\end{align*}
if $b=2((a-1)n+l)$ with some $a\in\{1,\ldots,1+\mathfrak{m}+r\}$ and $l\in\{1,\dots,n\}$.
Note that in the former case,
\begin{align*}
K_{2(1+\mathfrak{m}+r)nc+b}&={(\gamma_{\natural\sharp\flat})}^c \left(\gamma_{\natural\sharp\flat}\left[\frac{b-1}{2}\right]\right)[A_{V_0}, A_{U_l}]\\
&=\left[{(\gamma_{\natural\sharp\flat})}^c \left(\gamma_{\natural\sharp\flat}\left[\frac{b-1}{2}\right]\right)A_{V_0}, {(\gamma_{\natural\sharp\flat})}^c \left(\gamma_{\natural\sharp\flat}\left[\frac{b-1}{2}\right]\right)A_{U_l}\right]
\end{align*}
and in the latter case,
\begin{align*}
K_{2(1+\mathfrak{m}+r)nc+b}=\left[{(\gamma_{\natural\sharp\flat})}^c \left(\gamma_{\natural\sharp\flat}\left[\frac{b-2}{2}\right]\right)A_{U_l}, {(\gamma_{\natural\sharp\flat})}^c \left(\gamma_{\natural\sharp\flat}\left[\frac{b}{2}\right]\right)A_{V_0}\right].
\end{align*}
Then, Lemma~\ref{key0} implies that, for any $i\in\{1,\ldots,k\}$, 
$$\cdots\subsetneq J_{i,-1}^-\subsetneq  J_{i,0}^-\subsetneq J_{i,1}^-\subsetneq \cdots \subsetneq J_{i,2(1+\mathfrak{m}+r)n}^-\subsetneq J_{i,2(1+\mathfrak{m}+r)n+1}^-\subsetneq \cdots ,$$
$$\cdots\supsetneq J_{i,-1}^+\supsetneq J_{i,0}^+\supsetneq J_{i,1}^+\supsetneq \cdots \supsetneq J_{i,2(1+\mathfrak{m}+r)n}^+\supsetneq J_{i,2(1+\mathfrak{m}+r)n+1}^+\supsetneq \cdots.$$
Note that 
$$J_{i,0}^-\not\ni A_{V_0}\in J_{i,1}^-,$$
$$J_{i,2(1+\mathfrak{m}+r)n}^+\ni \gamma_{\natural\sharp\flat} A_{V_0} \not\in J_{i,2(1+\mathfrak{m}+r)n+1}^+.$$

The following is a variant for $C_{\Gamma,\Gamma_0}$ of 
\cite[Lemma 6.4]{KO2} for $C_{\Gamma}$. 
\begin{lemma}\label{key3}
The set 
$\{J_{i,d}\}_{i\in\{1,\ldots,k\}, d\in\{1,\ldots,2(1+\mathfrak{m}+r)n\}}$
is the set of all hyperplanes separating $A_{V_0}$ and $\gamma_{\natural\sharp\flat} A_{V_0}$. 
\end{lemma}
\noindent See Figure~\ref{sequence2_fig}.

\noindent To prove the lemma, we consider a path $\ell$ from $A_{V_0}$ to $\gamma_{\natural\sharp\flat} A_{V_0}$ 
that diagonally penetrates each of the cubes $K_{1},\ldots,K_{2(1+\mathfrak{m}+r)n}$ in order. 
The set of all hyperplanes intersecting the path $\ell$ is 
$\{J_{i,d}\}_{i\in\{1,\ldots,k\}, d\in\{1,\ldots,2(1+\mathfrak{m}+r)n\}}$. 

We now state the final lemma required to complete the proof of Theorem~\ref{main-4}.
Recall Definition~\ref{separatingdef} and Remark~\ref{separatingrem} for sequences of separating hyperplanes. 
The following is a variant for $C_{\Gamma,\Gamma_0}$ of 
\cite[Lemma 6.5]{KO2} for $C_{\Gamma}$. 

\begin{lemma}\label{key4}
\begin{enumerate}
\item For any $i\in\{1,\ldots,k\}$, 
the sequence of hyperplanes $$J_{i,1},\ldots,J_{i,2(1+\mathfrak{m}+r)n}$$ 
is a sequence of separating hyperplanes 
from $J_{i,0}$ to $J_{i,2(1+\mathfrak{m}+r)n+1}$, where $J_{i,0}$ and $J_{i,2(1+\mathfrak{m}+r)n+1}$ 
are of $v_{i,n}$-type and $v_{i,1}$-type, respectively 
(in particular, from $A_{V_0}$ to $\gamma_{\natural\sharp\flat} A_{V_0}$). 
See Figure~\ref{sequence2_fig}.
\item 
\begin{enumerate}
\item[(2-i)]
The sequence of hyperplanes 
\begin{align*}
J_{1,1},J_{1,2},\ldots,J_{1,2n}
\end{align*}
is a sequence of separating hyperplanes from 
$$\gamma_{\natural\sharp\flat}[0]A_{V_0}=A_{V_0} \text{ to } \gamma_{\natural\sharp\flat}[n]A_{V_0}=\gamma_{\natural}A_{V_0}.$$
\item[(2-ii)] The sequence of hyperplanes 
\begin{align*}
J_{1,2n+1},J_{1,2n+2},\ldots,J_{1,2(1+\mathfrak{m})n}
\end{align*}
is a sequence of separating hyperplanes from 
$$\gamma_{\natural\sharp\flat}[n]A_{V_0}=\gamma_{\natural}A_{V_0}\text{ to }\gamma_{\natural\sharp\flat}[(1+\mathfrak{m})n]A_{V_0}=\gamma_{\natural}\gamma_{\sharp}A_{V_0}.$$
\item[(2-iii-$a$)] For any $a\in \{1,\ldots, r\}$, the sequence of hyperplanes 
\begin{align*}
&J_{i_a,2(1+\mathfrak{m})n+2(a-1)n+1},J_{i_a,2(1+\mathfrak{m})n+2((a-1)n+1)},
\ldots,\\
&J_{i_a,2(1+\mathfrak{m})n+2((a-1)n+l_{\flat}(i_a,i_{a+1}))-1},
J_{i_{a+1},2(1+\mathfrak{m})n+2((a-1)n+l_{\flat}(i_a,i_{a+1}))},\\
&J_{i_{a+1},2(1+\mathfrak{m})n+2((a-1)n+l_{\flat}(i_a,i_{a+1}))+1},
\ldots,J_{i_{a+1},2(1+\mathfrak{m}+a)n}
\end{align*}
is a sequence of separating hyperplanes from 
$$\gamma_{\natural\sharp\flat}[(1+\mathfrak{m}+a-1)n]A_{V_0} \text{ to } \gamma_{\natural\sharp\flat}[(1+\mathfrak{m}+a)n]A_{V_0}.$$
See Figure~\ref{sequence_fig}.
\end{enumerate}
\item The sequence of hyperplanes composed of the sequences in (2-i), (2-ii), (2-iii-$1$), (2-iii-$2$), $\ldots$, (2-iii-$r$), in order, 
is a sequence of separating hyperplanes  
from $J_{1,0}$ to $J_{1,2(1+\mathfrak{m}+r)n+1}$, where $J_{1,0}$ and $J_{1,2(1+\mathfrak{m}+r)n+1}$ 
are of $v_{1,n}$-type and $v_{1,1}$-type, respectively 
(in particular, from $A_{V_0}$ to $\gamma_{\natural\sharp\flat} A_{V_0}$). 
Here, note that $i_1=1$. 
Moreover, the sequence contains a hyperplane of $v_i$-type for any $i\in\{1,\ldots,k\}$ and any $v_i\in V_i$. 
\end{enumerate}
\end{lemma}

\begin{figure}
\begin{center}
\includegraphics[width=14cm,pagebox=cropbox,clip]{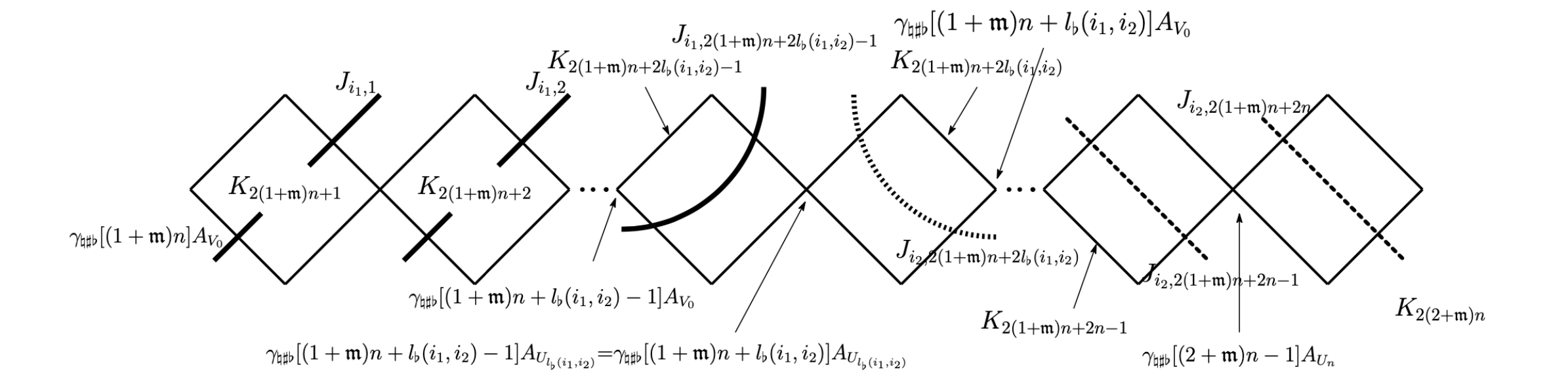}
\caption{The sequence of hyperplanes (2-iii-1) in Lemma~\ref{key4}
for the case where $\Gamma=\Gamma_0\ast\Gamma_1\ast\Gamma_2$. 
The two hyperplanes $J_{i_1,2(1+\mathfrak{m})n+2l_{\flat}(i_1,i_2)-1}$ and $J_{i_2,2(1+\mathfrak{m})n+2l_{\flat}(i_1,i_2)}$ do not intersect each other.
}\label{sequence_fig}
\end{center}
\end{figure}

\begin{proof}
(1) The assertion is clear from parts (1), (2), (3) and (4) of Lemma~\ref{key0}. 

(2) Parts (1), (2), (3) and (4) of Lemma~\ref{key0} and Lemma~\ref{key} imply that 
sequences in (2-i), (2-ii) and (2-iii-$a$) are sequences of separating hyperplanes.

(3) The sequence in the assertion is a sequence of separating hyperplanes.
We need to show that the sequence contains 
a hyperplane of $v_i$-type for any $i\in\{1,\ldots,k\}$ and any $v_i\in V_i$. 
Recall that $(V_{i_1},\dots,V_{i_r},V_{i_{r+1}})$ (\ref{pathtree}) is a closed path on the tree $T$. 
Thus, if an edge $(V_i,V_j)\in E(T)$ is contained in the closed path, then so is the inverse edge $(V_j,V_i)$. 
In addition, because $T$ is a spanning tree of $Q(\Gamma)$ and 
the closed path $(V_{i_1},\dots,V_{i_r},V_{i_{r+1}})$ passes through every vertex at least once, 
for any $i\in \{1,\ldots ,k\}$, $V_i$ is contained in the closed path as a vertex. 
Additionally, recall that $T$ has the root $V_1$ 
and that $i<j$ only if $V_i$ is not farther than $V_j$ from $V_1$ in $T$. 

Now, take any $i\in\{1,\ldots,k\}$. 
We consider the two cases of $i=1$ and $i\neq 1$. 

First, suppose that $i=1$. 
Note that $i_1=1$ and set $j=i_2$. 
Take $a\in \{1,2,\ldots, r\}$ such that $a\neq 1$, $i_a=1$, and $i_{a-1}=j$. 
Then, the sequence in the assertion contains the first half of the sequence of (2-iii-$1$):
\begin{align*}
J_{i_1,2(1+\mathfrak{m})n+1},J_{i_1,2(1+\mathfrak{m})n+2},\ldots,J_{i_1,2(1+\mathfrak{m})n+2l_{\flat}(i_1,i_{2})-1},
\end{align*}
which are of $v_{1,1}$-type, $v_{1,1}$-type, $\ldots$, 
$v_{1,l_{\flat}(1,j)}$-type, and the second half of the sequence of (2-iii-$(a-1)$):
\begin{align*}
&J_{i_{a},2(1+\mathfrak{m})n+2((a-2)n+l_{\flat}(i_{a-1},i_{a}))},J_{i_{a},2(1+\mathfrak{m})n+2((a-2)n+l_{\flat}(i_{a-1},i_{a}))+1},
\ldots,\\
&J_{i_{a},2(1+\mathfrak{m})n+2(a-1)n},
\end{align*}
which are of $v_{1,l_{\flat}(j,1)}$-type, $v_{1,l_{\flat}(j,1)+1}$-type, $\ldots$, 
$v_{1,n}$-type. 
Note that $l_{\flat}(j,1)=l_{\flat}(1,j)$ by Lemma~\ref{closed}. 
Take any vertex $v_1\in V_1$. 
Then, there exists some $l\in\{1,\ldots,n\}$ such that $v_{1,l}=v_1$ by Lemma~\ref{closed}. 
Hence, the sequence  in the assertion contains a hyperplane of $v_1$-type. 

Next, suppose that $i\neq 1$.  
Take the smallest $a\in \{1,\ldots, r\}$ such that $i_a=i$. 
Because $i_1=1$ and $i\neq 1$, we have $a\neq 1$. 
We set $j=i_{a-1}$. 
Then, we have $a'\in \{1,\ldots, r\}$ such that 
$a\le a'$, $i_{a'}=i$, and $i_{a'+1}=j$. 
The sequence in the assertion contains the second half of the sequence of (2-iii-$(a-1)$):
\begin{align*}
&J_{i_{a},2(1+\mathfrak{m})n+2((a-2)n+l_{\flat}(i_{a-1},i_{a}))},J_{i_{a},2(1+\mathfrak{m})n+2((a-2)n+l_{\flat}(i_{a-1},i_{a}))+1},
\ldots,\\
&J_{i_{a},2(1+\mathfrak{m})n+2(a-1)n},
\end{align*}
which are of $v_{i,l_{\flat}(j,i)}$-type, $v_{i,l_{\flat}(j,i)+1}$-type, $\ldots$, 
$v_{i,n}$-type, and the first half of the sequence of (2-iii-$a'$):

\begin{align*}
&J_{i_{a'},2(1+\mathfrak{m})n+2(a'-1)n+1},J_{i_{a'},2(1+\mathfrak{m})n+2((a'-1)n+1)},
\ldots,\\
&J_{i_{a'},2(1+\mathfrak{m})n+2((a'-1)n+l_{\flat}(i_{a'},i_{a'+1}))-1},
\end{align*}
which are of $v_{i,1}$-type, $v_{i,1}$-type, $\ldots$, 
$v_{i,l_{\flat}(j,i)}$-type. 
Note that $l_{\flat}(j,i)=l_{\flat}(i,j)$ by Lemma~\ref{closed}. 
Take any vertex $v_i\in V_i$. 
Then, there exists some $l\in\{1,\ldots,n\}$ such that $v_{i,l}=v_i$ by Lemma~\ref{closed}. 
Hence, the sequence in the assertion contains a hyperplane of $v_i$-type. 
\end{proof}

We can now complete the proof of Theorem~\ref{main-4}. 
\begin{proof}[Proof of $(1)\Rightarrow(2)$ in Theorem~\ref{main-4}]
We show $(1)\Rightarrow(2)$ in Theorem~\ref{main-4}. 
Consider $\gamma_{\natural\sharp\flat}\in A_\Gamma$ defined by (\ref{gammaelement}) 
and hyperplanes 
$$J:=J_{1,0} \text{ and } J':=J_{1,2(1+\mathfrak{m}+r)n+1},$$ 
which are of $v_{1,n}$-type and $v_{1,1}$-type, respectively. 
We now verify conditions (i), (ii) and (iii) in Theorem~\ref{criterion}. 

(i) $\gamma_{\natural\sharp\flat}$ skewers $(J,J')$. Indeed, it is clear that 
$$J^+\supsetneq \gamma_{\natural\sharp\flat}^{-1}(J'^+)\supsetneq\gamma_{\natural\sharp\flat}(J^+)\supsetneq J'^+\supsetneq \gamma_{\natural\sharp\flat}^2(J^+).$$

(ii) We show that $J$ and $J'$ are strongly separated. 
Take any hyperplane $H$ with $J\cap H\neq \emptyset$. 
When $H$ is of $v_i$-type for some $i$ and $v_i\in V_i$, 
take a hyperplane $H'$ of $v_i$-type 
separating $A_{V_0}$ and $\gamma_{\natural\sharp\flat} A_{V_0}$ 
such that $J\cap H'=\emptyset$ by part (3) of Lemma~\ref{key4}. 
Then, $H\cap H'=\emptyset$ by Remark~\ref{label}. 
Because $J\subset H'^-, H'^+\supset J'$,  $J\cap H\neq \emptyset$ and $H\cap H'=\emptyset$, we have 
$J'\cap H=\emptyset$. 

(iii) We show $Stab(J)\cap Stab(J')=\{1\}$. 
We divide the proof into three parts:
(iii-1) we show that 
$Stab(J)\cap Stab(J')\subset A_{V_0}$;
(iii-2) we show that 
$Stab(J)\cap Stab(J')\subset \gamma_{\natural\sharp\flat} A_{V_0}\gamma_{\natural\sharp\flat}^{-1}$;
(iii-3) we show that $Stab(J)\cap Stab(J')=\{1\}$ using (iii-1) and (iii-2).

(iii-1) 
Note that for any $i\in\{1,\ldots,k\}$ and any $v_i\in V_i$, 
we have at least one sequence of separating hyperplanes 
$P'_1,\ldots,P'_{M'}$ from $A_{V_0}$ to $\gamma_{\natural\sharp\flat} A_{V_0}$ 
such that $P'_1$ is of $v_i$-type and $P'_{M'}$ is of $v_{1,n}$-type. 
For example, we can take such a sequence 
by considering a subsequence of 
the sequence in part (3) of Lemma~\ref{key4}. 
For any $i\in\{1,\ldots,k\}$ and any $v_i\in V_i$, 
we can take a longest sequence of separating hyperplanes 
$P_1,\ldots,P_M$ from $A_{V_0}$ to $\gamma_{\natural\sharp\flat} A_{V_0}$ 
such that $P_1$ is of $v_i$-type and $P_M$ is of $v_{1,n}$-type, 
where 
$$A_{V_0} \in P_1^-, P_1^+\supsetneq P_2^+\supsetneq\cdots 
\supsetneq P_{M-1}^+\supsetneq P_M^+\ni \gamma_{\natural\sharp\flat} A_{V_0}$$ 
by taking decompositions into appropriate connected components 
$$C_{\Gamma,\Gamma_0}\!\setminus\!\!\setminus P_1
=P_1^-\sqcup P_1^+,\ldots,C_{\Gamma,\Gamma_0}\!\setminus\!\!\setminus P_M=P_M^-\sqcup P_M^+.$$
Note that $P_1,\ldots,P_M\in \{J_{j,d}\}_{j\in\{1,\ldots,k\}, d\in\{1,\ldots,2(1+\mathfrak{m}+r)n\}}$ 
by Lemma~\ref{key3}. 
By noting $P_M\in \{J_{1,d}\}_{d\in\{1,\ldots,2(1+\mathfrak{m}+r)n\}}$ 
and applying part (1) of Lemma~\ref{key4} for the case $i=1$, 
we have $P_M^+ \supsetneq J'^+$. 

Now assume that $Stab(J)\cap Stab(J')\not\subset A_{V_0}$. 
Take $g\in Stab(J)\cap Stab(J')$ with $g\not\in A_{V_0}$.
Note that $g^{-1}\in Stab(J)\cap Stab(J')$. 
Then, we have a hyperplane $H$ of $v_i$-type for some $i\in\{1,\ldots,k\}$ 
and some $v_i\in V_i$, 
separating $A_{V_0}$ and $gA_{V_0}$. 
We take a longest sequence of separating hyperplanes 
$P_1,\ldots,P_M$ from $A_{V_0}$ to $\gamma_{\natural\sharp\flat} A_{V_0}$ 
such that $P_1$ is of $v_i$-type and $P_M$ is of $v_{1,n}$-type. 
Then we have 
$$g\gamma_{\natural\sharp\flat} A_{V_0}, g^{-1}\gamma_{\natural\sharp\flat} A_{V_0}\in P_M^+$$
by $P_M^+ \supsetneq J'^+$. 
We take a connected component 
$H^-$ of $C_{\Gamma,\Gamma_0}\!\setminus\!\!\setminus H$ such that $A_{V_0}\in H^-$. 
Then the other connected component $H^+$ satisfies $gA_{V_0}\in H^+$. 
Because $H\cap J\neq \emptyset$, we have $H\cap J'=\emptyset$ by (ii). 
Thus, $H$ cannot separate 
$\gamma_{\natural\sharp\flat} A_{V_0}$ and $g\gamma_{\natural\sharp\flat} A_{V_0}$. 
Hence we have either 

\noindent (iii-1-$a$) 
$A_{V_0}, \gamma_{\natural\sharp\flat} A_{V_0},g \gamma_{\natural\sharp\flat} A_{V_0} \in H^-  \text{ and } gA_{V_0} \in H^+$ or 

\noindent (iii-1-$b$)
$A_{V_0}\in H^- \text{ and } gA_{V_0}, \gamma_{\natural\sharp\flat} A_{V_0},g \gamma_{\natural\sharp\flat} A_{V_0} \in H^+$.

Assume that case (iii-1-$a$) occurs. Then $H$ does not separate 
$A_{V_0}$ and $\gamma_{\natural\sharp\flat} A_{V_0}$ and thus 
$H\notin\{J_{i,d}\}_{d\in\{1,\ldots,2(1+\mathfrak{m}+r)n\}}$ by Lemma~\ref{key3}. 
Hence, $H$ and $P_1$ are different hyperplanes of the same $v_i$-type, and thus 
they cannot intersect by Remark~\ref{label}. 
Then we have either $H^-\supsetneq P_1^+$ or $H^+\supsetneq P_1^+$. 
However, $H^+\supsetneq P_1^+$ cannot occur because 
$\gamma_{\natural\sharp\flat} A_{V_0}\notin H^+$ and $\gamma_{\natural\sharp\flat} A_{V_0} \in P_1^+$. 
Therefore, 
$$H^-\supsetneq P_1^+.$$ 
Then we have 
$$gA_{V_0} \in H^+, H^-\supsetneq P_1^+\supsetneq P_2^+\supsetneq\cdots 
\supsetneq P_{M-1}^+\supsetneq P_M^+\ni g\gamma_{\natural\sharp\flat} A_{V_0}.$$ 
Then, $Q_1=H, Q_2=P_1,\ldots,Q_{M+1}=P_M$ is a sequence of separating hyperplanes 
from $gA_{V_0}$ to $g\gamma_{\natural\sharp\flat} A_{V_0}$ 
such that $Q_1$ is of $v_i$-type and $Q_{M+1}$ is of $v_{1,n}$-type. 
Thus, $g^{-1}Q_1, g^{-1}Q_2,\ldots,g^{-1}Q_{M+1}$ is a sequence of separating hyperplanes 
from $A_{V_0}$ to $\gamma_{\natural\sharp\flat} A_{V_0}$ 
such that $g^{-1}Q_1$ is of $v_i$-type and $g^{-1}Q_{M+1}$ is of $v_{1,n}$-type, 
which contradicts the fact that the sequence $P_1,\ldots,P_M$ is the longest. 

Next, assume that the case (iii-1-$b$) occurs. Then $H$ does not separate 
$gA_{V_0}$ and $g\gamma_{\natural\sharp\flat} A_{V_0}$; that is, 
$g^{-1}H$ does not separate 
$A_{V_0}$ and $\gamma_{\natural\sharp\flat} A_{V_0}$ and thus 
$g^{-1}H\notin\{J_{i,d}\}$ $({d\in\{1,\ldots,2(1+\mathfrak{m}+r)n\}})$ by Lemma~\ref{key3}. 
Hence, $g^{-1}H$ and $P_1$ are different hyperplanes of the same $v_i$-type, and thus 
they cannot intersect by Remark~\ref{label}. 
Then we have either $g^{-1}H^+\supsetneq P_1^+$ or $g^{-1}H^-\supsetneq P_1^+$. 
However $g^{-1}H^-\supsetneq P_1^+$ can not occur because 
$\gamma_{\natural\sharp\flat} A_{V_0}\notin g^{-1}H^-$ and $\gamma_{\natural\sharp\flat} A_{V_0} \in P_1^+$. 
Therefore 
$$g^{-1}H^+\supsetneq P_1^+.$$ 
Then we have 
$$g^{-1}A_{V_0} \in g^{-1}H^-, g^{-1}H^+\supsetneq P_1^+\supsetneq P_2^+\supsetneq\cdots 
\supsetneq P_{M-1}^+\supsetneq P_M^+\ni g^{-1}\gamma_{\natural\sharp\flat} A_{V_0}.$$ 
Then, $Q_1=H, Q_2=gP_1,\ldots,Q_{M+1}=gP_M$ is a sequence of separating hyperplanes 
from $A_{V_0}$ to $\gamma_{\natural\sharp\flat} A_{V_0}$ 
such that $Q_1$ is of $v_i$-type and $Q_{M+1}$ is of $v_{1,n}$-type, 
which contradicts the fact that the sequence $P_1,\ldots,P_M$ is the longest. 

We now see that $Stab(J)\cap Stab(J')\subset A_{V_0}$. 

(iii-2) We show $Stab(J)\cap Stab(J')\subset \gamma_{\natural\sharp\flat} A_{V_0}\gamma_{\natural\sharp\flat}^{-1}$.
This can be shown in a similar way as (iii-1).
Indeed we reach a contradiction by assuming $Stab(J)\cap Stab(J')\not\subset \gamma_{\natural\sharp\flat} A_{V_0}\gamma_{\natural\sharp\flat}^{-1}$.

(iii-3) We show that $Stab(J)\cap Stab(J')=\{1\}$. 
Take $g\in Stab(J)\cap Stab(J')$.
Then, by (iii-1) and (iii-2), we have $g\in A_{V_0}\cap \gamma_{\natural\sharp\flat} A_{V_0}\gamma_{\natural\sharp\flat}^{-1}$. 
Let $\ell$ be the geodesic from  $A_{V_0}$ to $\gamma_{\natural\sharp\flat} A_{V_0}$, 
which diagonally penetrates each of the cubes $K_{1},\ldots,K_{2(1+\mathfrak{m}+r)n}$ in order. 
We note that the set of all hyperplanes intersecting the geodesic $\ell$ is 
$\{J_{i,d}\}_{i\in\{1,\ldots,k\}, d\in\{1,\ldots,2(1+\mathfrak{m}+r)n\}}$. 
Since $g$ fixes $A_{V_0}$ and $\gamma_{\natural\sharp\flat} A_{V_0}$, $g$ fixes the geodesic $\ell$ and thus all vertices on the geodesic $\ell$.
Take any $w\in V_0$. 
Recall that $V_0=W(1)\sqcup \cdots \sqcup W(\mathfrak{h})$, where $W(h)$ $(1\leq h\leq \mathfrak{h})$ is defined in $(\ref{W(h)})$.
Let $h\in \{1,\ldots,\mathfrak{h}\}$ be the height of $w$; that is, the number satisfying $w\in W(h)$. 

Suppose that $g\in A_{V_0\setminus\{w'\in V_0\mid w'\prec w\}}=A_{\{w''\in V_0\mid w''\succeq w\}}$. 
We show that $g\in A_{\{w''\in V_0\mid w''\succ w\}}$. 

\begin{figure}
\begin{center}
\includegraphics[width=12cm,pagebox=cropbox,clip]{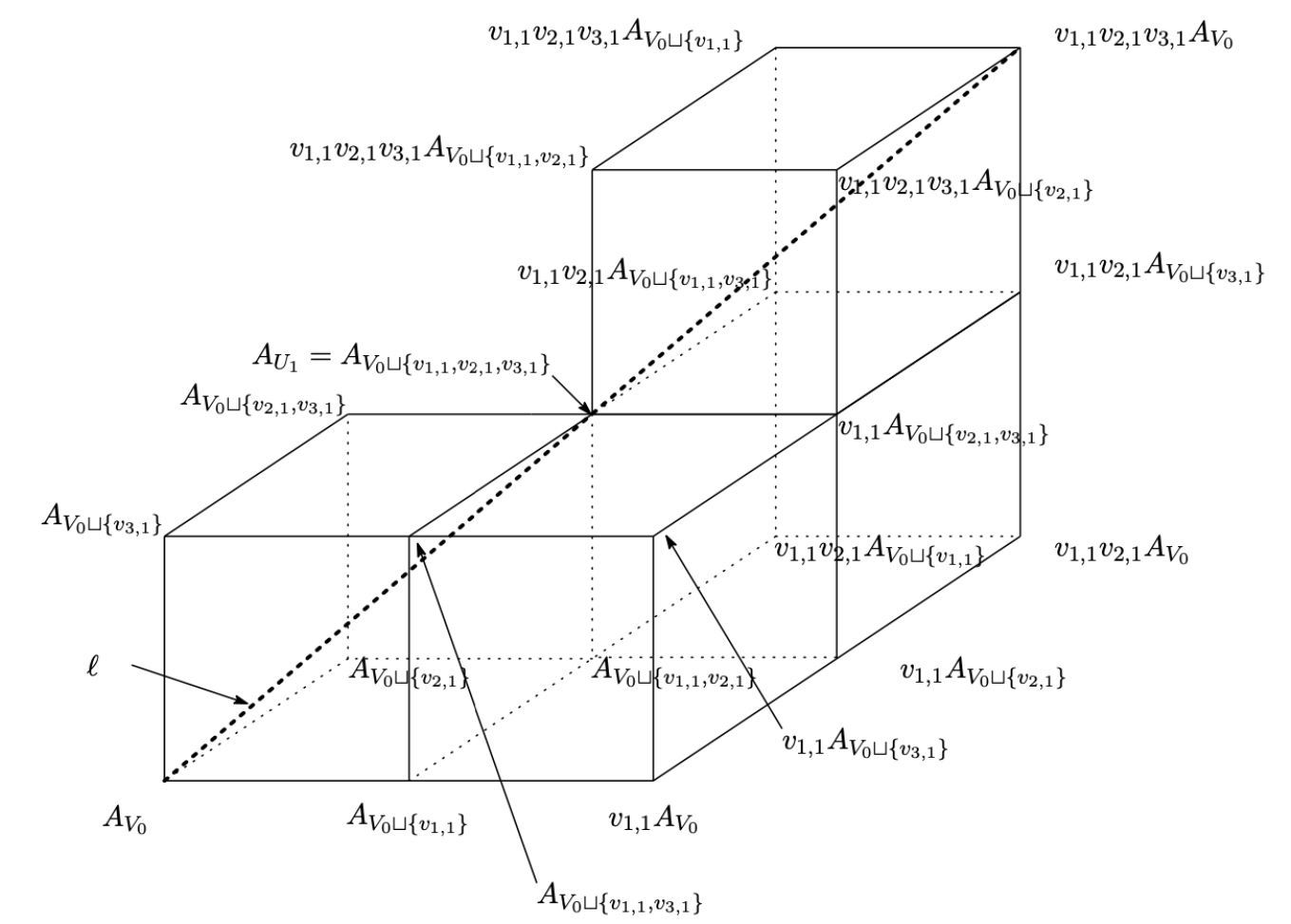}
\caption{Cubes penetrated by $\ell$ when $k=3$ and $l_{\sharp}(\mathfrak{j})=1$. 
In this case, $g$ fixes cubes $[A_{V_0}, A_{U_1}]$ and $[v_{1,1}v_{2,1}v_{3,1}A_{V_0}, A_{U_1}]$, which are penetrated by $\ell$.
Then $g$ fixes three edges $[A_{V_0\sqcup\{v_{1,1},v_{2,1}\}}, A_{U_1}]$, $[A_{V_0\sqcup \{v_{1,1},v_{3,1}\}}, A_{U_1}]$ and $[v_{1,1}A_{V_0\sqcup\{v_{2,1},v_{3,1}\}}, A_{U_1}]$,
and thus three squares $[A_{V_0\sqcup\{v_{1,1}\}}, A_{U_1}]$, $[v_{1,1}A_{V_0\sqcup\{v_{2,1}\}}, A_{U_1}]$ and $[v_{1,1}A_{V_0\sqcup\{v_{3,1}\}}, A_{U_1}]$.
Hence $g$ fixes the cube $[v_{1,1}A_{V_0}, A_{U_1}]$.
By a similar argument, $g$ fixes the cube $[v_{1,1}v_{2,1}A_{V_0}, A_{U_1}]$. 
Therefore, $g$ fixes not only $A_{V_0}$ and $v_{1,1}v_{2,1}v_{3,1}A_{V_0}$, but also $v_{1,1}A_{V_0}$ and $v_{1,1}v_{2,1}A_{V_0}$.}\label{iii-3_fig}
\end{center}
\end{figure}

(iii-3-$a$) We consider the case where $h=1$; that is, $w\in W(1)$. 
Then we have $v_{i_{\sharp}(w),l_{\sharp}(w)}\in V_{\ast}$. 
We consider a vertex on $\ell$ of the form
\begin{equation}\label{intermediatevertex_a}
\lambda_1\cdots\lambda_{l_{\sharp}(w)-1}
\lambda_{l_{\sharp}(w)}A_{V_0},
\end{equation}
which is fixed by $g$.
By definition, in $(\ref{intermediatevertex_a})$,
$$\lambda_{l_{\sharp}(w)}=v_{1,l_{\sharp}(w)}\cdots v_{i_{\sharp}(w),l_{\sharp}(w)}\cdots v_{k,l_{\sharp}(w)}.$$
Since the action of $A_{\Gamma}$ on $C_{\Gamma,\Gamma_0}$ preserves the cube structure, all cubes penetrated by $\ell$ are fixed by $g$.
Therefore, $g$ fixes 
$$\lambda_1\cdots\lambda_{l_{\sharp}(w)-1}v_{1,l_{\sharp}(w)}\cdots v_{i_{\sharp}(w),l_{\sharp}(w)}A_{V_0},$$ 
see Figure~\ref{iii-3_fig}.
Then 
$$g\in (\lambda_1\cdots\lambda_{l_{\sharp}(w)-1}v_{1,l_{\sharp}(w)}\cdots v_{i_{\sharp}(w),l_{\sharp}(w)})A_{V_0}(\lambda_1\cdots\lambda_{l_{\sharp}(w)-1}v_{1,l_{\sharp}(w)}\cdots v_{i_{\sharp}(w),l_{\sharp}(w)})^{-1}.$$
By $g\in A_{\{w''\in V_0\mid w''\succeq w\}}$ and Lemma~\ref{sharp_lem} (1), $\lambda_1\cdots \lambda_{l_{\sharp}(w)-1}v_{1,l_{\sharp}(w)}\cdots v_{i_{\sharp}(w)-1,l_{\sharp}(w)}$ commutes with $g$, see Figure~\ref{W_fig}.
Therefore, $g\in v_{i_{\sharp}(w),l_{\sharp}(w)}A_{V_0}v_{i_{\sharp}(w),l_{\sharp}(w)}^{-1}$, and thus
$g\in A_{V_0}\cap v_{i_{\sharp}(w),l_{\sharp}(w)}A_{V_0}v_{i_{\sharp}(w),l_{\sharp}(w)}^{-1}$.
According to Lemma~\ref{Godelle},
$A_{V_0}\cap v_{i_{\sharp}(w),l_{\sharp}(w)}A_{V_0}v_{i_{\sharp}(w),l_{\sharp}(w)}^{-1}\subset A_{V_0\setminus \{w\}}$.
Hence 
$$g\in A_{\{w''\in V_0\mid w''\succeq w\}}\cap A_{V_0\setminus\{w\}}=A_{\{w''\in V_0\mid w''\succ w\}}.$$

(iii-3-$b$) We consider the case where $w\in W(h)$ where $h\geq 2$. 
Then we have 
$$(w(h), w(h-1),\ldots w(0))=(w, \pi(w),\ldots \pi^h(w))$$
in $(\ref{w(h)_path})$.
Then we have $\mathfrak{j}\in \{1,\ldots, \mathfrak{m}\}$ such that 
$w_{\mathfrak{j}}=w(h-1)=\pi(w)\in W(h-1)$.
We consider a vertex on $\ell$ of the form
\begin{equation}\label{intermediatevertex_b}
\gamma_{\natural}\gamma_\sharp(1)\cdots\gamma_\sharp(\mathfrak{j}-1)
\lambda_1\cdots\lambda_{l_{\sharp}(w_{\mathfrak{j}})-1}
\lambda_{\sharp}(\mathfrak{j})A_{V_0},
\end{equation}
which is fixed by $g$.
We note that 
$\lambda_{\sharp}(\mathfrak{j})$ equals 
\begin{align*}
&v_{1,l_{\sharp}(w_{\mathfrak{j}})}\cdots  v_{i_{\sharp}(w_{\mathfrak{j}})-1,l_{\sharp}(w_{\mathfrak{j}})}
v_{i_{\sharp}(w_{\mathfrak{j}})+1,l_{\sharp}(w_{\mathfrak{j}})}\cdots v_{k,l_{\sharp}(w_{\mathfrak{j}})}w_{\mathfrak{j}}(h-1)\cdots w_{\mathfrak{j}}(1)w_{\mathfrak{j}}(0)\\
&=v_{1,l_{\sharp}(w_{\mathfrak{j}})}\cdots  v_{i_{\sharp}(w_{\mathfrak{j}})-1,l_{\sharp}(w_{\mathfrak{j}})}
v_{i_{\sharp}(w_{\mathfrak{j}})+1,l_{\sharp}(w_{\mathfrak{j}})}\cdots v_{k,l_{\sharp}(w_{\mathfrak{j}})}w(h-1)\cdots w(1)w(0).
\end{align*}
Hence 
\begin{align*}
&\gamma_{\natural}\gamma_{\sharp}(1)\cdots \gamma_{\sharp}(\mathfrak{j}-1)\lambda_1\cdots \lambda_{l_{\sharp}(w_{\mathfrak{j}})-1}\lambda_{\sharp}(\mathfrak{j})\\
=&\left(\gamma_{\natural}\gamma_{\sharp}(1)\cdots \gamma_{\sharp}(\mathfrak{j}-1)\lambda_1\cdots \lambda_{l_{\sharp}(w_{\mathfrak{j}})-1}\right)\cdot \\
&\left(v_{1,l_{\sharp}(w_{\mathfrak{j}})}\cdots  v_{i_{\sharp}(w_{\mathfrak{j}})-1,l_{\sharp}(w_{\mathfrak{j}})}
v_{i_{\sharp}(w_{\mathfrak{j}})+1,l_{\sharp}(w_{\mathfrak{j}})}\cdots v_{k,l_{\sharp}(w_{\mathfrak{j}})}w(h-1)\cdots w(1)w(0)\right).
\end{align*}
Note that  
$\gamma_{\natural}\gamma_{\sharp}(1)\cdots \gamma_{\sharp}(\mathfrak{j}-1)$
is a product of elements of $V_{\ast}\sqcup \{w_1,\ldots w_{\mathfrak{j}-1}\}$.
By Lemma~\ref{sharp_lem} (2) with $w_1\prec \cdots \prec w_{\mathfrak{j}-1}\prec w_{\mathfrak{j}}=\pi(w)$, 
any $w''\succeq w$ commutes with 
$$\gamma_{\natural}\gamma_\sharp(1)\cdots\gamma_\sharp(\mathfrak{j}-1)
\lambda_1\cdots\lambda_{l_{\sharp}(w_{\mathfrak{j}})-1}v_{1,l_{\sharp}(w_{\mathfrak{j}})}\cdots  v_{i_{\sharp}(w_{\mathfrak{j}})-1,l_{\sharp}(w_{\mathfrak{j}})}
v_{i_{\sharp}(w_{\mathfrak{j}})+1,l_{\sharp}(w_{\mathfrak{j}})}\cdots v_{k,l_{\sharp}(w_{\mathfrak{j}})}. $$
Hence so does $g$.
Then $g\in (w(h-1)\cdots w(0))
A_{V_0}(w(h-1)\cdots w(0))^{-1}$.
Next, by noting $w(0)\prec w(1) \prec\cdots \prec w(h-1) \prec w(h)=w$, 
we can show $g\in A_{\{w''\in V_0\mid w''\succ w\}}$.
Indeed, it follows from Lemma~\ref{Godelle} that 
$g$ belongs to 
\begin{align*}
&A_{\{w''\in V_0\mid w''\succeq w\}}
\\
&\cap(w(h-1)\cdots w(0))
A_{V_0}(w(h-1)\cdots w(0))^{-1}
\\
=&A_{\{w''\in V_0\mid w''\succeq w\}}\cap A_{V_0}
\\
&\cap(w(h-1)\cdots w(1)w(0))
A_{V_0}(w(h-1)\cdots w(1)w(0))^{-1}
\\
=&A_{\{w''\in V_0\mid w''\succeq w\}}
\cap (w(h-1)\cdots w(1))A_{V_0}(w(h-1)\cdots w(1))^{-1}
\\
&\cap(w(h-1)\cdots w(1)w(0))
A_{V_0}(w(h-1)\cdots w(1)w(0))^{-1}
\\
= &A_{\{w''\in V_0\mid w''\succeq w\}}
\\
&\cap
(w(h-1)\cdots w(1))(A_{V_0}\cap w(0)A_{V_0}w(0)^{-1})
(w(h-1)\cdots w(1))^{-1}
\\
\subset&A_{\{w''\in V_0\mid w''\succeq w\}}
\\
&\cap
(w(h-1)\cdots w(1))
A_{V_0\setminus\{w(1)\}}(w(h-1)\cdots w(1))^{-1}
\\
=&A_{\{w''\in V_0\mid w''\succeq w\}}
\cap A_{V_0\setminus\{w(1)\}}
\\
&\cap
(w(h-1)\cdots w(2)w(1))
A_{V_0\setminus\{w(1)\}}(w(h-1)\cdots w(2)w(1))^{-1}
\\
=&A_{\{w''\in V_0\mid w''\succeq w\}}
\cap (w(h-1)\cdots w(2))A_{V_0\setminus\{w(1)\}}(w(h-1)\cdots w(2))^{-1}
\\
&\cap
(w(h-1)\cdots w(2)w(1))
A_{V_0\setminus\{w(1)\}}(w(h-1)\cdots w(2)w(1))^{-1}
\\
=&A_{\{w''\in V_0\mid w''\succeq w\}}\\
&\cap
(w(h-1)\cdots w(2))
(A_{V_0\setminus\{w(1)\}}\cap w(1)A_{V_0\setminus\{w(1)\}}
w(1)^{-1})(w(h-1)\cdots w(2))^{-1}
\\
\subset&A_{\{w''\in V_0\mid w''\succeq w\}}
\\
&\cap
(w(h-1)\cdots w(2))
A_{V_0\setminus\{w(1),w(2)\}}
(w(h-1)\cdots w(2))^{-1}
\\
&\cdots
\\
\subset&A_{\{w''\in V_0\mid w''\succeq w\}}
\\
&\cap
w(h-1)
A_{V_0\setminus\{w(1),\ldots,w(h-1)\}}
w(h-1)^{-1}
\\
=&A_{\{w''\in V_0\mid w''\succeq w\}}\\
&\cap
(A_{V_0\setminus\{w(1),\ldots,w(h-1)\}}\cap w(h-1)A_{V_0\setminus\{w(1),\ldots,w(h-1)\}}w(h-1)^{-1})
\\
\subset&A_{\{w''\in V_0\mid w''\succeq w\}}
\\
&\cap
A_{V_0\setminus\{w(1),\ldots, w(h-1),w(h)\}}
\\
=&A_{\{w''\in V_0\mid w''\succ w\}}. 
\end{align*} 
In the above deformation, all $\subset$ are consequences of Lemma~\ref{Godelle},
and the final equality follows from $w(h)=w$.
The other equalities are trivial deformations.

Since $\prec$ is a total order on a finite set $V_0$,
by applying (iii-3-$a$) or (iii-3-$b$) for $w\in V_0$ in order from the smallest element of $V_0$ to the largest one, we have $g\in A_{\emptyset}=\{e\}$ and thus $g=e$.
\end{proof}

\begin{proof}[Proof of Corollaries~\ref{main2} and \ref{main3}]
Let $A_\Gamma$ be an irreducible Artin group associated with $\Gamma$ that is not a clique. 
Then, Theorem~\ref{main} implies that $A_\Gamma$ is acylindrically hyperbolic. 
Therefore, the center of $A_\Gamma$ is finite by acylindrical hyperbolicity (\cite[Corollary 7.3]{Osin}). 
Proposition~\ref{finite normal subgroup} then implies Corollary~\ref{main2}. 
We can show Corollary~\ref{main3} in a similar way.  
\end{proof}

\begin{remark}\label{center_rem}
There are several versions of the so-called center conjecture.
Note that the center of any irreducible Artin group of spherical type is known to be infinite cyclic \cite{BS}, \cite{MR422673}.
Then, as one version, the center conjecture containing this known fact can be stated as follows:
the center of any irreducible Artin group of spherical type and of infinite type
is infinite cyclic and trivial, respectively.
Since the center of any possibly reducible Artin group with each irreducible factor satisfying
the center conjecture is trivial or torsion-free,
the last part of Corollary~\ref{main2} clearly implies the following:
in the setting of Corollary~\ref{main2},
if each irreducible factor of $A_{\Gamma_0}$ satisfies the center conjecture, then
$A_\Gamma$ satisfies the center conjecture.
This claim is independently given in \cite{JM} by a different method.
\end{remark}

\section{Example}\label{example_sec}
Let $\Gamma$ be the graph with $V(\Gamma)=\{x,y,u_1,u_2,v_1,v_2\}$ as shown in Figure~\ref{example_fig}.
In this section, we explain the construction of $\gamma_{\natural\sharp\flat}$ for this example.

\begin{remark}
Theorems~\ref{main} and \ref{main-4} can be applied to Artin groups associated with $\Gamma$ in Figures~\ref{new_ex_fig} and \ref{new_ex_fig2}.
However, these examples may be too simple to illustrate the construction of $\gamma_{\natural\sharp\flat}$.
Indeed, in the leftmost defining graph in Figure~\ref{new_ex_fig}, $\gamma_{\flat}=e\in A_{\Gamma}$. 
For the middle and right $\Gamma$, $\gamma_{\sharp}=\gamma_{\flat}=e\in A_{\Gamma}$.
Also, for all $\Gamma$ in Figure~\ref{new_ex_fig2},
$\gamma_{\sharp}=e\in A_{\Gamma}$.
The graph $\Gamma$ in Figure~\ref{example_fig} is one of the smallest that has non-trivial $\gamma_{\sharp}$ and $\gamma_{\flat}$.
\end{remark}

We note that $\Gamma$ is irreducible; that is, $\Gamma^t$ is connected.
Then we have the join decomposition $\Gamma=\Gamma_0\ast \Gamma_{\ast}=\Gamma_0\ast(\Gamma_1\ast \Gamma_2)$,
where we set subgraphs $\Gamma_i$ $(i=0,1,2)$ of $\Gamma$ by
\begin{align}
&V_0:=V(\Gamma_0)=\{x,y\},\ E_0:=E(\Gamma_0)=\{(x,y), \ (y,x)\},\\
&V_1:=V(\Gamma_1)=\{u_1,u_2\},\ E_1:=E(\Gamma_1)=\emptyset,\\
&V_2:=V(\Gamma_2)=\{v_{1},v_{2}\},\ E_2:=E(\Gamma_2)=\emptyset.
\end{align}
Clearly, $\Gamma_0$ is the maximal clique factor, and
$\Gamma_1$ and $\Gamma_2$ are join-indecomposable factors. 
Note that $k=2$.
We give an explicit construction of $\gamma_{\natural\sharp\flat}=\gamma_{\natural}\gamma_{\sharp}\gamma_{\flat}$ for this $\Gamma$ following Section~\ref{main_proof}.

First, we construct $\gamma_{\natural}$.
In our $\Gamma$, $Q(\Gamma,\Gamma_0)$ is the tree consisting of one edge and two vertices $V_1$ and $V_2$ (see (\ref{QGamma})).
Here, the spanning tree $T$ coincides with $Q(\Gamma, \Gamma_0)$.
We choose $V_1$ as the root of $T$.
Take paths $p_{1,2}$ and $p_{2,1}$ in $\Gamma^t$ as follows:
\begin{align}
p_{1,2}=(w_{1,2,0},w_{1,2,1},w_{1,2,2})=(u_{1},x,v_{1}),\\
p_{2,1}=(w_{2,1,0},w_{2,1,1},w_{2,1,2})=(v_{1},x,u_{1})
\end{align}
(see (\ref{p_ij}) and (\ref{p_ji})).
Their end vertices are
\begin{align}
s_{1,2}=w_{1,2,0}=u_{1}\in V_1,\ t_{1,2}=w_{1,2,2}=v_{1}\in V_2,\\
s_{2,1}=w_{2,1,0}=v_{1}\in V_2,\ t_{2,1}=w_{2,1,2}=u_{1}\in V_1,
\end{align}
respectively (see (\ref{st_ij})).
Set $\tau_{1,2}, \tau_{2,1}\in A_\Gamma$ as
\begin{align}\label{tau_ex}
\tau_{1,2}=w_{1,2,0}w_{1,2,1}w_{1,2,2}=u_{1}xv_{1},\\
\tau_{2,1}=w_{2,1,0}w_{2,1,1}w_{2,1,2}=v_{1}xu_{1}
\end{align}
(see (\ref{tau_ij})).
We fix a closed path on $T=Q(\Gamma,\Gamma_0)$;
\begin{align}
(V_{i_1},V_{i_2}, V_{i_3})=(V_1, V_2, V_1)
\end{align}
(see (\ref{pathtree})).
We note that the length $r$ of the path is $2$.

According to Lemma~\ref{closed}, we construct $n\in \mathbb{N}$, $M_{\natural}\in \mathcal{M}(2,n;V(\Gamma))$ and a map $l_{\flat}$.
Take closed paths of length $n=2$ on $\Gamma_1^c$, $\Gamma_2^c$:
\begin{align}
(u_1,u_2,u_1),\ (v_1,v_2,v_1), respectively.
\end{align}

Set 
\begin{align}
M_{\natural}=
\begin{pmatrix}
u_1&u_2\\
v_1&v_2
\end{pmatrix},
\end{align}
and
\begin{align}\label{l_natural_ex}
l_{\flat}(1,2)=l_{\flat}(2,1)=1.
\end{align}
Then, 
\begin{align}
M_{\natural}[1,l_{\flat}(1,2)]=u_1,\ 
M_{\natural}[2,l_{\flat}(2,1)]=v_{1}.
\end{align}
We have 
\begin{align}\label{gammaelement^0_ex}
\gamma_{\natural}=u_1v_1u_2v_2,
\end{align}
see $(\ref{gammaelement^0})$.

Next, with $M_{\natural}\in \mathcal{M}(2,2;V(\Gamma))$ and $l_{\flat}$,
we give an explicit construction of $\gamma_\flat$. 
With $(\ref{l_natural_ex})$, $M'_{\flat}(a)\in \mathcal{M}(2,2;A_{\Gamma})$ $(a=1,2)$ are
\begin{equation}\label{M_natural_a_1_ex}
M'_{\flat}(1)=M'_{\flat}(2)=
\begin{pmatrix}
u_1^{-1}& e\\
v_1^{-1} & e
\end{pmatrix},
\end{equation}
see $(\ref{M_natural_a_1})$.
Using $(\ref{tau_ex})$, supplementary matrices $M''_{\flat}(a)\in \mathcal{M}(2,2;A_{\Gamma})$ $(a=1,2)$ are
\begin{equation}\label{M_natural_1_2_ex}
M''_{\flat}(1)=
\begin{pmatrix}
\tau_{1,2}& e\\
e & e
\end{pmatrix}
=
\begin{pmatrix}
u_1xv_1& e\\
e & e
\end{pmatrix}
\end{equation}
and 
\begin{equation}\label{M_natural_2_2_ex}
M''_{\flat}(2)=
\begin{pmatrix}
\tau_{2,1}& e\\
e & e
\end{pmatrix}
=
\begin{pmatrix}
v_1xu_1& e\\
e & e
\end{pmatrix},
\end{equation}
see $(\ref{M_natural_a_2})$.
With $(\ref{M_natural_a_1_ex}), (\ref{M_natural_1_2_ex})$ and 
$(\ref{M_natural_2_2_ex})$, $M_{\flat}(a)=M''_{\flat}(a)\odot M'_{\flat}(a)\odot M_{\natural}$ $(a=1,2)$ are
\begin{equation}
M_{\flat}(1)=
\begin{pmatrix}
u_1xv_1& u_2\\
e & v_2
\end{pmatrix},
\end{equation}
and 
\begin{equation}
M_{\flat}(2)=
\begin{pmatrix}
v_1xu_1& u_2\\
e & v_2
\end{pmatrix},
\end{equation}
see $(\ref{M_natural_a})$.
Then
\begin{equation}
M_{\flat}=
\begin{pmatrix}
u_1xv_1& u_2&v_1xu_1& u_2\\
e & v_{2}&e & v_{2}
\end{pmatrix},
\end{equation}
see $(\ref{M_natural})$.
We have
\begin{equation}\label{gammaelement'_ex}
\gamma_\flat=(u_1xv_1)(u_{2}v_{2})(v_1xu_1)(u_{2}v_{2}),
\end{equation}
see $(\ref{gammaelement'})$.

We next explain the construction of $\gamma_{\sharp}$.
According to $(\ref{W(h)})$,
\begin{equation}
W(1)=\{x\},\ W(2)=\{y\}.
\end{equation}
Note that 
$d_{\Gamma^t}(x,V_{\ast})=1$, $d_{\Gamma^t}(y,V_{\ast})=2$.
Then $V_0=W(1)\sqcup W(2)$ and $\mathfrak{h}=2$.
Here, $d_{\Gamma^t}(x,V_{\ast})=d_{\Gamma^t}(x,u_1)=d_{\Gamma^t}(x,v_1)$.
Therefore,
\begin{equation}
(i_{\sharp}(x),l_{\sharp}(x))=(i_{\sharp}(y),l_{\sharp}(y))=(1,1)
\end{equation}
(see $(\ref{i_flat})$ and $(\ref{l_flat})$), and
\begin{equation}\label{index_total-order_ex}
w_1=w_{\mathfrak{m}}=x
\end{equation}
where
\begin{equation}\label{density_total-order_ex}
\mathfrak{m}=\#\left(W(1)\sqcup \cdots \sqcup W(\mathfrak{h}-1)\right)=\#W(1)=1
\end{equation}
(see $(\ref{index_total-order})$ and $(\ref{density_total-order})$). 
Then 
\begin{equation}\label{l_flat_ex}
l_{\sharp}(w_1)=l_{\sharp}(x)=1.
\end{equation}
With $(\ref{l_flat_ex})$, $M'_{\sharp}(1)\in \mathcal{M}(2,2;A_{\Gamma})$ is
\begin{equation}\label{M_flat_j_1_ex}
M'_{\sharp}(1)=
\begin{pmatrix}
u_{1}^{-1}& e\\
e & e
\end{pmatrix},
\end{equation}
see $(\ref{M_flat_j_2})$.
Using $(\ref{l_flat_ex})$, $M''_{\sharp}(1)\in \mathcal{M}(2,2;A_{\Gamma})$ is
\begin{equation}\label{M_flat_1_2_ex}
M''_{\sharp}(1)=
\begin{pmatrix}
e& e\\
w_1(1)w_1(0) & e
\end{pmatrix}
=
\begin{pmatrix}
e& e\\
xu_{1}& e
\end{pmatrix}, 
\end{equation}
see $(\ref{M_flat_j_1})$.
With $(\ref{M_flat_j_1_ex}), (\ref{M_flat_1_2_ex})$, $M_{\sharp}(1)=M_{\natural}\odot M'_{\sharp}(1)\odot M''_{\sharp}(1)$ is
\begin{equation}
M_{\sharp}(1)=
\begin{pmatrix}
e& u_2\\
v_1xu_1 &v_{2}
\end{pmatrix},
\end{equation}
see $(\ref{M_flat_j})$.
Then
\begin{equation}
M_{\sharp}=M_{\sharp}(1)=
\begin{pmatrix}
e& u_2\\
v_1xu_1 &v_{2}
\end{pmatrix},
\end{equation}
see $(\ref{M_flat})$, and
we have
\begin{equation}\label{gammaelement''_ex}
\gamma_{\sharp}=(v_{1}xu_{1})(u_{2}v_{2}),
\end{equation}
see $(\ref{gammaelement''})$.

Then
\begin{equation}\label{gammaelement'''_ex}
\gamma_{\natural\sharp\flat}=\gamma_{\natural}\gamma_{\sharp}\gamma_{\flat}=(u_1v_1u_2v_2)(v_{1}xu_{1}u_{2}v_{2})(u_1xv_1u_{2}v_{2}v_1xu_1u_{2}v_{2})
\end{equation}
is the WPD element constructed in Section~\ref{main_proof}, 
where $\gamma_{\natural}$, $\gamma_{\sharp}$ and $\gamma_{\flat}$ are $(\ref{gammaelement^0_ex})$, $(\ref{gammaelement''_ex})$ and $(\ref{gammaelement'_ex})$.
See $(\ref{gammaelement''})$.

\section*{Acknowledgments}
The authors would like to thank Antoine Goldsborough, Nicolas Vaskou, Kasia Jankiewicz and MurphyKate Montee for useful conversations.
Especially the authors thank Antoine Goldsborough and Nicolas Vaskou
for their pointing out to us that Theorem~\ref{main} has an application to random Artin groups.



\begin{thebibliography}{99}
\bibitem{BBF} Mladen Bestvina, Ken Bromberg and Koji Fujiwara. 
Constructing group actions on quasi-trees and applications to mapping class groups. 
Publ. Math. Inst. Hautes \'{E}tudes Sci., 122:1--64, 2015.
\bibitem{MR1914565} Mladen Bestvina and Koji Fujiwara. 
Bounded cohomology of subgroups of mapping class groups. 
Geom. Topol., 6:69--89, 2002.
\bibitem{BB}
Anders Bjorner and Francesco Brenti. 
Combinatorics of Coxeter Groups. 
Springer, 2005.  
\bibitem{BS} Egbert Brieskorn and Kyoji Saito, Artin-Gruppen und Coxeter-Gruppen, Invent. Math. 17 (1972), 245--271. 
\bibitem{MR2367021} Brian H. Bowditch. 
Tight geodesics in the curve complex. 
Invent. Math., 171(2):281--300, 2008.
\bibitem{MR1770639} Thomas Brady and Jonathan P. McCammond. 
Three-generator Artin groups of large type are biautomatic. 
J. Pure Appl. Algebra, 151(1):1--9, 2000.
\bibitem{BH} Martin R. Bridson and Andr\'{e} Haefliger. 
Metric spaces of non-positive curvature, 
volume 319 of Grundlehren der Mathematischen Wissenschaften [Fundamental Principles of Mathematical Sciences]. Springer-Verlag, Berlin, 1999.
\bibitem{MR323910} Egbert Brieskorn and Kyoji Saito. 
Artin-Gruppen und Coxeter-Gruppen. Invent. Math., 17:245--271, 1972.
\bibitem{Calvez} Matthieu Calvez. Euclidean Artin-Tits groups are acylindrically hyperbolic. Groups Geom. Dyn. 16 (2022), no. 3, 963--983.
\bibitem{MR3719080} Matthieu Calvez and Bert Wiest. 
Acylindrical hyperbolicity and Artin-Tits groups of spherical type. 
Geom. Dedicata, 191:199--215, 2017.
\bibitem{MR2827012} Pierre-Emmanuel Caprace and Michah Sageev. 
Rank rigidity for CAT(0) cube complexes. 
Geom. Funct. Anal., 21(4):851--891, 2011.
\bibitem{Charney2008PROBLEMSRT} Ruth Charney. 
Problems related to Artin groups. 
\url{http://people.brandeis.edu/~charney/papers/Artin_probs.pdf}. 2008.
\bibitem{CMM} Ruth Charney, Alexandre Martin, Rose Morris-Wright, Acylindrical hyperbolicity for Artin groups with a visual splitting. preprint, arXiv:2404.11393.
\bibitem{Charney} Ruth Charney and Rose Morris-Wright. 
Artin groups of infinite type: trivial centers and acylindrical hyperbolicity. 
Proc. Amer. Math. Soc., 147(9):3675--3689, 2019.
\bibitem{MR3966610} Indira Chatterji and Alexandre Martin. 
A note on the acylindrical hyperbolicity of groups acting on CAT(0) cube complexes, In Beyond hyperbolicity. 
volume 454 of London Math. Soc. Lecture Note Ser., pages 160--178. Cambridge Univ. Press, Cambridge, 2019.
\bibitem{MR3589159} F. Dahmani, V. Guirardel and D. Osin. 
Hyperbolically embedded subgroups and rotating families in groups acting on hyperbolic spaces. 
Mem. Amer. Math. Soc., 245(1156):v+152, 2017.
\bibitem{MR422673} Pierre Deligne. 
Les immeubles des groupes de tresses g\'{e}n\'{e}ralis\'{e}s. 
Invent. Math., 17:273--302, 1972.
\bibitem{Genevois} Anthony Genevois. 
Acylindrical hyperbolicity from actions on CAT(0) cube complexes: a few criteria. 
New York J. Math. 25, 1214--1239, 2019.
\bibitem{Godelle} Eddy Godelle. 
On parabolic subgroups of Artin-Tits groups. 
arXiv:2207.06528.
\bibitem{MR3203644} Eddy Godelle and Luis Paris. 
Basic questions on Artin-Tits groups. 
In Configuration spaces, volume 14 of CRM Series, pages 299--311. Ed. Norm., Pisa, 2012.
\bibitem{Goldman} Katherine Goldman. 
The $K(\pi,1)$ conjecture and acylindrical hyperbolicity for relatively extra-large Artin groups. 
Algebraic and Geometric Topology, to appear. arXiv:2211.16391.
\bibitem{GV} Antoine Goldsborough, Nicolas Vaskou. Random Artin groups. preprint, arXiv:2301.04211. 
\bibitem{haettel2019xxl} Thomas Haettel. XXL type Artin groups are CAT(0) and acylindrically hyperbolic. Ann. Inst. Fourier (Grenoble) 72 (2022), no. 6, 2541--2555.
\bibitem{MR2390326} Ursula Hamenst\"{a}dt. 
Bounded cohomology and isometry groups of hyperbolic spaces. 
J. Eur. Math. Soc. (JEMS), 10(2):315--349, 2008.
\bibitem{HO20}
Jingyin Huang and Damian Osajda. 
Large-type Artin groups are systolic. 
Proc. Lond. Math. Soc. (3) 120, no. 1, 95--123, 2020.
\bibitem{HO21}
Jingyin Huang and Damian Osajda. 
Helly meets Garside and Artin. 
Invent. Math. 225, no. 2, 395--426, 2021.
\bibitem{JM} Kasia Jankiewicz and MurphyKate Montee. Centers of Artin groups defined on cones. arXiv:2406.06480.
\bibitem{KO} Motoko Kato and Shin-ichi Oguni. 
Acylindrical hyperbolicity of Artin-Tits groups associated with triangle-free graphs and cones over square-free bipartite graphs. 
Glasg. Math. J. 64, no. 1, 51--64, 2022. 
\bibitem{KO2} Motoko Kato and Shin-ichi Oguni. 
Acylindrical hyperbolicity of Artin groups associated with graphs that are not cones. 
Groups Geom.\ Dyn.\ 18, no. 4, 1291--1316, 2024.
\bibitem{MR3192368} Sang-Hyun Kim and Thomas Koberda. 
The geometry of the curve graph of a right-angled Artin group. 
Internat. J. Algebra Comput., 24(2):121--169, 2014.
\bibitem{alex2019acylindrical} 
Alexandre Martin and Piotr Przytycki. 
Acylindrical actions for two-dimensional Artin groups of hyperbolic type. Int. Math. Res. Not. IMRN 2022, no. 17, 13099--13127. 
\bibitem{MR2240393} Koji Nuida. On the direct indecomposability of infinite irreducible Coxeter groups and the isomorphism problem of Coxeter groups. Comm. Algebra 34 (2006), no. 7, 2559--2595. 
\bibitem{Osin} Denis Osin. 
Acylindrically hyperbolic groups. 
Trans. Amer. Math. Soc., 368(2):851--888, 2016.
\bibitem{Osin2} Denis Osin. 
Groups acting acylindrically on hyperbolic spaces. 
Proceedings of the International Congress of Mathematicians—Rio de Janeiro 2018. Vol. II. 
Invited lectures, 919--939, World Sci. Publ., Hackensack, NJ, 2018.
\bibitem{Paris} Luis Paris. 
Parabolic subgroups of Artin groups. 
J. Algebra 196, no. 2, 369--399, 1997.
\bibitem{MR2333366} Luis Paris, Irreducible Coxeter groups. Internat. J. Algebra Comput. 17 (2007), no. 3, 427--447.
\bibitem{Sageev} Michah Sageev. 
Ends of group pairs and non-positively curved cube complexes. 
Proc. London Math. Soc. (3) 71, no. 3, 585--617, 1995. 
\bibitem{SS} S.\ Sakai, M.\ Sakuma, Combinatorial local convexity implies convexity in finite dimensional CAT(0) cubed complexes, preprint, arXiv:2302.10500.
\bibitem{Shi} Jian-Yi Shi. 
The Enumeration of Coxeter Elements. 
Journal of Algebraic Combinatorics 6, 161--171, 1997. 
\bibitem{Tits} Jacques Tits. 
Normalisateurs de tores. I. Groupes de Coxeter \'{e}tendus. 
J. Algebra 4 , 96--116, 1966. 
\bibitem{van} 
Harm van der Lek. 
The homotopy type of complex hyperplane complements. 
PhD thesis, Nijmegen, 1983.
\bibitem{Vaskou} 
Nicolas Vaskou. 
Acylindrical hyperbolicity for Artin groups of dimension 2. 
Geom. Dedicata 216, no. 1, Paper No. 7, 28 pp, 2022. 
\end{thebibliography}
\end{document}